\documentclass[a4paper,11pt]{amsart}
\usepackage{amsmath,amssymb,amsthm,mathtools}
\usepackage{enumitem}
\usepackage{hyperref}
\usepackage{mathrsfs}
\usepackage{tikz}
\usepackage{comment}
\usepackage{todonotes}
\usepackage{esint}

\hypersetup{
	colorlinks=true,
	linkcolor=blue,
	citecolor=blue,
	urlcolor=blue
}

\numberwithin{equation}{section}

\theoremstyle{definition}
\newtheorem{definition}{Definition}[section]
\newtheorem{remark}[definition]{Remark}

\theoremstyle{plain}
\newtheorem{theorem}[definition]{Theorem}
\newtheorem{proposition}[definition]{Proposition}
\newtheorem{corollary}[definition]{Corollary}
\newtheorem{lemma}[definition]{Lemma}

\newcommand{\R}{\mathbb{R}}

\newcommand{\Sph}{\mathbb{S}}
\newcommand{\dist}{\operatorname{dist}}

\newcommand{\kqh}{k_\Omega}

\newcommand{\CAT}{\mathrm{CAT}}
\newcommand{\inj}{\operatorname{injrad}}

\newcommand{\euc}{\mathrm{euc}}
\newcommand{\Hess}{\operatorname{Hess}}

\newcommand{\Cut}{\operatorname{Cut}}

\newcommand{\injrad}{\operatorname{injrad}}
\newcommand{\conjrad}{\operatorname{conjrad}}

\title{Quasihyperbolic domains are CAT(2)}
\author{Toni Ikonen}
\author{Abhishek Pandey}

\address{Department of Mathematics and Statistics\\ P. O. Box 68 (Pietari Kalmin katu 5) \\ FI-00014 University of Helsinki\\ Finland}
\email{toni.ikonen@helsinki.fi}

\address{
	Department of Mathematics and Statistics\\
	University of Jyväskylä\\
	P.O. Box 35\\
	FI-40014 University of Jyväskylä\\
	Finland
}
\email{pandey.a.pandey@jyu.fi}


\thanks{T. Ikonen was supported by the Swiss National Science Foundation grant 212867 and the Research Council of Finland, project number 374557. A. Pandey was supported by the Research Council of Finland, project number 360505.}

\begin{document}

\begin{abstract}
We prove that every proper subdomain of the Euclidean space equipped with its quasihyperbolic metric is CAT(2). The CAT property leads to a positive resolution of Väisälä's uniqueness, prolongation, quasihyperbolic convexity, and local geodesic conjectures on quasihyperbolic geodesics affirmatively on all dimensions, extending the two-dimensional results by Herron. The sectional curvature bound is optimal as demonstrated by the complement of two points in dimensions at least three.

We also show that in any quasihyperbolic domain, quasihyperbolic spheres up to radius $\pi/\sqrt{2}$ are $\mathcal{C}^{1,\frac{1}{2}}$-diffeomorphic to the Euclidean sphere. Consequently, we answer a question by Gehring and Vuorinen on the regularity of such spheres. Similar techniques lead to a positive resolution of Väisälä's conjecture on the Euclidean convexity of quasihyperbolic balls up to radius $\arctan(\sqrt{2})/\sqrt{2}$. 

We establish that in all dimensions, the quasihyperbolic metric is CAT(0) on convex domains. This leads to a positive answer to a question by Martio and Väisälä on the quasihyperbolic convexity of quasihyperbolic balls on such domains. 

Finally, we prove that a variable Alexandrov curvature lower bound for a quasihyperbolic domain self-improves to a global lower bound of $-1$ and is equivalent to the concavity of the domain.
\end{abstract}

\maketitle
	\pagestyle{myheadings}
	\markboth{Quasihyperbolic domains are CAT(2)}{Toni Ikonen and Abhishek Pandey}
	
\section{Introduction and main results}
\subsection{Background and main results}
In this manuscript, we consider the geometry of proper domains $\Omega \subsetneq \R^n$ equipped with the quasihyperbolic metric $k_\Omega$; the pair $(\Omega,k_\Omega)$ is a \emph{quasihyperbolic domain}. The quasihyperbolic geometry of the domain is sensitive to the structure of its boundary since $k_\Omega$ is constructed by minimizing the weighted Riemannian length
\begin{equation*}
    \ell_{k_\Omega}(\gamma)
    =\int_\gamma\frac{ds}{\delta_\Omega(x)},
\end{equation*}
over rectifiable curves, where $\delta_{\Omega} \colon \Omega \to (0,\infty)$ is the distance to the complement. That is,
$$\delta_\Omega(x)=\dist(x,\R^n \setminus \Omega)\coloneqq\inf_{p\in \R^n \setminus \Omega}|x-p|.$$ Quasihyperbolic domains were first introduced by Gehring and Palka \cite{Gehring-Palka-1976} for the purpose of analyzing quasiconformally homogeneous domains. Shortly after that, Gehring and Osgood classified uniform domains using the quasihyperbolic metric \cite{Gehring-Osgood-1979-uniform-domains-and-the-quasihyperbolic-metric}. Since then, these metrics have found numerous applications in Sobolev analysis and on uniformization problems \cite{Jones-1980-extension-theorems-for-BMO,Jones-1981-quasiconformal-mappings-and-extendability-of-functions-in-Sobolev-spaces,herron-koskela-1991-uniform-sobolev-extension-and-quasiconformal-circle-domains,Jones-Smirnov-2000-removability-theorems,Bonk-Heinonen-Koskela-2001-uniformizing-gromov-hyperbolic-spaces,Ntalampekos-Younsi-2020-rigidity-theorems-for-circle-domains,Karafyllia-Ntalampekos-2024}; see the survey \cite{Koskela-1998-old-and-new-on-the-quasihyperbolic-metric} and the monograph \cite{Bonk-Heinonen-Koskela-2001-uniformizing-gromov-hyperbolic-spaces} for further background. 

Since the metric space $(\Omega,k_\Omega)$ is a complete and locally compact length space, the metric Hopf--Rinow theorem implies the existence of length-minimizing curves joining any pair of points, called \emph{quasihyperbolic geodesics}. The large scale properties of $(\Omega,k_\Omega)$, and connections to Gromov hyperbolicity, have been studied in \cite{Bonk-Heinonen-Koskela-2001-uniformizing-gromov-hyperbolic-spaces,Balogh-Buckley-2003-geometric-characterizations-of-gromov-hyperbolicity,Buckley-2020-quasihyperbolic-geodesics-are-hyperbolic-quasi-geodesics}.

Regarding local properties, Väisälä proposed several conjectures in~\cite{Vaisala2009}.

\medskip

\textit{Uniqueness conjecture.}
There is a universal constant $c_U>0$ such that if $x,y\in \Omega$ and $\kqh(x,y)<c_U$,
then there is only one quasihyperbolic geodesic joining $x$ to $y$.

\medskip

\textit{Prolongation conjecture.}
There is a universal constant $c_P>0$ such that if $\gamma\colon x\curvearrowright y$
is a quasihyperbolic geodesic with $\ell_{\kqh}(\gamma)<c_P$, then there is a
quasihyperbolic geodesic $\gamma_1\colon x\curvearrowright y_1$ of length $c_P$ such that
$\gamma\subset \gamma_1$.

\medskip

\textit{Local geodesic conjecture.}
There is a universal constant $c_{LG}>0$ such that a local quasihyperbolic geodesic of quasihyperbolic length $\leq c_{LG}$ is a quasihyperbolic geodesic.

\medskip

\textit{Quasihyperbolic convexity conjecture.}
There is a universal constant $c_{QH}>0$ such that every quasihyperbolic ball $B_{k_\Omega}(p,r)$ is quasihyperbolically convex for every $r<c_{QH}$. That is, if $x, y \in B_{k_\Omega}(p,r)$, any quasihyperbolic geodesic joining $x$ to $y$ is contained in $B_{k_\Omega}(p,r)$.

\medskip

Väisälä \cite{Vaisala2009} resolved the first three conjectures positively in two dimensions. More recently, Herron \cite[Corollary D]{Herron-2020} proved sharp versions of all four conjectures in two dimensions. Herron achieved this by proving that any planar quasihyperbolic domain has its Alexandrov curvature bounded from above by $0$ and is $\mathrm{CAT}(1)$ \cite[Theorems A and C]{Herron-2020}; we recall that a proper length space $X$ is $\mathrm{CAT}(\kappa)$ for $\kappa > 0$ if its Alexandrov curvature is bounded from above by $\kappa$ \emph{and} geodesics with length $< D_\kappa \coloneqq \pi/\sqrt{\kappa}$ are unique; if $\kappa \leq 0$, the definition is otherwise the same but $D_\kappa = \infty$.

We extend Herron's result to higher dimensions as follows.

\begin{theorem}\label{thm:universal-cat2}
    For any domain $\Omega \subsetneq \R^n$ and $n \geq 3$, the quasihyperbolic domain $(\Omega,k_\Omega)$ is $\CAT(2)$.
\end{theorem}
The following is an immediate consequence of $\mathrm{CAT}(2)$ theory, see Section~\ref{section-consequences-of-theorems} for details.
\begin{theorem}\label{theorem-resolution-of-conjectures}
    All four conjectures of Väisälä hold in all dimensions with constants independent of the dimension: we may take $c_U = c_P = c_{LG} = D_2$ and $c_{QH} = D_2/2$.
\end{theorem}
Theorem~\ref{thm:universal-cat2} is sharp.
\begin{theorem}\label{thm:sharp-cat2}
    If $n\geq3$, the quasihyperbolic domain $( \Omega, k_\Omega )$ for $\Omega=\R^n\setminus\{-e_1,e_1\}$ is not (locally) $\mathrm{CAT}(\kappa)$ for any $\kappa < 2$.
\end{theorem}
\begin{remark}\label{remark-no-curvature-lower-bounds-or-smooth-structure}
    We highlight that Theorem~\ref{thm:universal-cat2} truly concerns non-smooth Riemannian and metric geometry through two main points.
    
    First, the quasihyperbolic domain $(\Omega,k_\Omega)$ from Theorem~\ref{thm:sharp-cat2} shows that the sectional curvature can concentrate on a lower-dimensional set. Indeed, on the one hand, if $Q = \{ x_1 = 0 \}$, restricting $k_\Omega$ to $\Omega \setminus Q$ defines a distance locally isometric to $\R \times \Sph^{n-1}$ and thus $( \Omega \setminus Q, k_\Omega )$ is locally $\mathrm{CAT}(1)$. On the other hand, $( \Omega, k_\Omega )$ is not even locally $\CAT(\kappa)$ for $\kappa <2$ by Theorem~\ref{thm:sharp-cat2}. In this sense, the sectional curvature concentrates to $Q$; see Section~\ref{section-sharpness-of-the-curvature-bound} for further discussion on the example.

    Second, we provide two examples showing the non-smoothness of the quasihyperbolic metric. In particular, we illustrate the taxonomy of sets on which $\delta_{\Omega}$ is non-differentiable. In case $M$ is a compact smooth $n$-manifold, by the Whitney embedding theorem, $M$ embeds smoothly in $\R^{2n+1}$. Now, a small enough tubular neighbourhood $\Omega \subset \R^{2n+1}$ of $M$ is a disk bundle over $M$ for which the set of non-differentiability points of $\delta_\Omega$ coincides with $M$. On the other extreme, by a result due to Zamfirescu \cite{zamfirescu-1990-the-nearest-point-mapping-is-single-valued-nearly-everywhere}, the non-differentiability points of $\delta_\Omega$ is dense in $\Omega$ for many domains.
\end{remark}
\begin{remark}\label{remark-gcba-space-examples}
    In light of Theorem~\ref{thm:universal-cat2} and Remark~\ref{remark-no-curvature-lower-bounds-or-smooth-structure}, quasihyperbolic domains and their covering spaces provide new examples of nonsmooth $\mathrm{CAT}(\kappa)$ spaces homeomorphic to manifolds, and, in particular, GCBA spaces in the sense of Lytchak and Nagano \cite{lytchak-nagano-2019-gcba-spaces}.
    
    In contrast to general GCBA topological manifolds (see \cite{lytchak-nagano-2022-topological-regularity-of-spaces-with-an-upper-curvature-bound}), a simple argument using the definition of $k_\Omega$ shows that the space of directions at every point of a quasihyperbolic domain is isometric to the sphere $\Sph^{n-1}$. Thus every point of a quasihyperbolic domain is regular in the sense of \cite{lytchak-nagano-2019-gcba-spaces}. See Remark~\ref{remark-local-dc-regularity} for further discussion.
\end{remark}

While Theorem~\ref{thm:universal-cat2} shows that an Alexandrov curvature upper bound is universal on quasihyperbolic domains, the following theorem demonstrates that a lower bound is rather rigid.
\begin{theorem}\label{theorem-cbb-equivalent-to-concavity}
    For a domain $\Omega \subsetneq \R^n$ and $n \geq 2$, the following are equivalent.
    \begin{enumerate}
        \item $\Omega$ is concave.
        \item $\delta_{\Omega}$ is differentiable at every point of $\Omega$.
        \item $(\Omega, k_\Omega)$ satisfies a local variable Alexandrov curvature lower bound.
        \item $(\Omega,k_\Omega)$ is $\mathrm{CBB}(-1)$.
    \end{enumerate}
    Under these assumptions, the differential of $\delta_{\Omega}$ is locally Lipschitz in $\Omega$ and $\delta_{\Omega}$ is convex.
\end{theorem}
Recall that a domain $\Omega$ is concave if $\R^n \setminus \Omega$ is convex. For (3), a local variable bound means that each point is contained in an open neighbourhood satisfying an Alexandrov curvature lower bound for a constant, possibly depending on the neighbourhood. In contrast, $\mathrm{CBB}(-1)$ refers to a uniform lower bound of $-1$ \cite{AlexanderKapovitchPetrunin}.

While Theorem~\ref{theorem-cbb-equivalent-to-concavity} classifies the quasihyperbolic geometry of concave domains, convex domains have received more attention. In fact, Martio and Väisälä \cite{Vaisala-2005-quasihyperbolic-geodesics-in-convex-domains,Martio-Vaisala-2011} proved that in a convex domain, quasihyperbolic geodesics are unique and extend to a globally minimizing geodesic. These results also follow from our next result.
\begin{theorem}\label{thm:convex-cat0}
    If $\Omega\subsetneq\R^n$ is convex, then $(\Omega,k_\Omega)$ is $\CAT(0)$.
\end{theorem}
When $n = 2$, Theorem~\ref{thm:convex-cat0} is a consequence of the Cartan--Hadamard theorem and \cite[Theorem A]{Herron-2020} (or \cite[Corollary B]{Herron-2020}). As far as the authors are aware, Theorem~\ref{thm:convex-cat0} was previously unknown in higher dimensions. We note that the convexity of the domain is sufficient but not necessary for the $\CAT(0)$ conclusion; see Section~\ref{example:nonconvex-cat0-example} for a concave example. Martio and Väisälä asked in \cite[Question 2.14]{Martio-Vaisala-2011} whether quasihyperbolic balls are quasihyperbolically convex in convex domains. Theorem \ref{thm:convex-cat0} implies a positive answer to this question by standard $\CAT(0)$ theory; see~Remark~\ref{remark-vaisalas-conjectures}.

Our final two results concern the regularity of quasihyperbolic balls. The question on the regularity of quasihyperbolic balls goes back to discussions by Gehring and Vuorinen from the 1970s and progress has been limited until recently when Klén, Rasila and Talponen proved in \cite{KlenRasilaTalponen2017} that quasihyperbolic balls in convex domains of various Banach spaces, including Hilbert spaces, are $\mathcal{C}^1$-regular; for convex domains in Hilbert spaces, the $\mathcal{C}^1$-regularity also follows from \cite[Theorem~5.13]{vaisala-2007-quasihyperbolic-geometry-of-domains-in-hilbert-spaces} and \cite[Theorem~2.13]{Martio-Vaisala-2011}. Building upon Theorem~\ref{thm:universal-cat2}, we address the general case in Euclidean domains as follows.
\begin{theorem}\label{theorem-c1-regular-spheres}
     For any quasihyperbolic domain $( \Omega, k_\Omega )$ and $p \in \Omega$, there exists a $\mathcal{C}^{1,\frac{1}{2}}_{\text{loc}}$-regular diffeomorphism $$\Psi = ( \Psi_1, \Psi_2 ) \colon B_{k_\Omega}(p,D_2) \setminus \{p\} \to ( 0, D_2 ) \times \Sph^{n-1}$$ such that $\Psi_1(x) = k_{\Omega}(x,p)$ for $x \in B_{k_\Omega}(p,D_2) \setminus \{p\}$. Moreover, if $( \Omega, k_\Omega )$ is $\mathrm{CAT}(\kappa)$, the conclusion holds with $D_2$ replaced by $D_\kappa$.
\end{theorem}
The statement is sharp in the following sense: if $n \geq 2$, then $( \R^n \setminus \{0\}, k_{ \R^n \setminus \{0\} } )$ is isometric to $\R \times \mathbb{S}^{n-1}$ and thus $\mathrm{CAT}(1)$. The distance from a point to its cut locus in $\R \times \mathbb{S}^{n-1}$ is exactly $D_1 = \pi$ and the corresponding fiber of the distance function is not even a topological manifold.

To formulate our last theorem, we recall the following conjecture by Väisälä, also from \cite{Vaisala2009}.

\medskip

\textit{Convexity conjecture.} There is a universal constant $c_C>0$ such that the quasihyperbolic ball
$\overline{B}_{\kqh}(p,r)$ is strictly convex for all $r<c_C$. That is, for each $x,y \in \overline{B}_{\kqh}(p,r)$, the Euclidean line segment joining $x$ to $y$ lies in $B_{\kqh}(p,r) \cup \{x,y\}$.

\medskip

Väisälä proved that a planar quasihyperbolic domain satisfies the convexity conjecture with the constant $c_{C} = 1$ which is sharp by explicit examples due to Klén \cite{Klen-2008}. Luiro has announced in \cite{Luiro-2015} that the convexity conjecture holds in all dimensions with the constant $c_C = 1/100$; unfortunately \cite{Luiro-2015} remains unpublished. Using similar techniques as in the proof of Theorem~\ref{theorem-c1-regular-spheres}, we prove the following.
\begin{theorem}\label{theorem-small-balls-euclidean-convex}
     Let $( \Omega, k_\Omega )$ be a quasihyperbolic domain. Then the following holds.
     \begin{enumerate}
         \item if $(\Omega,k_\Omega)$ is $\CAT(\kappa)$ for $0 < \kappa \leq 2$, then the conclusion of the convexity conjecture holds for the constant $\arctan( \sqrt{\kappa} )/\sqrt{\kappa}$.
         \item if $\kappa \leq 0$ or $n = 2$, the conclusion holds up to radius one.
     \end{enumerate}
     In particular, $c_{C} = \arctan(\sqrt{2})/\sqrt{2} \approx 0.6755 $ in all dimensions.
\end{theorem}
The constant $c_{C}$ improves on Luiro's announced constant. Aside from the planar case, the authors do not know if the convexity constant is sharp.

\subsection{Outline of the manuscript}
We recall prerequisites from Riemannian and conformal geometry in Section~\ref{section-preliminaries-riemannian-geometry} and from metric and quasihyperbolic geometry in Section~\ref{section-preliminaries-metric-geometry}.

In Section~\ref{section-qh-metric-in-finite-punctures}, we establish the $\mathrm{CAT}(2)$ property for quasihyperbolic domains whose complements are finite. The proof involves an approximation of the quasihyperbolic density by suitable smooth densities which satisfy a sectional curvature upper bound $\leq 2$ for $n \geq 3$ (and $\leq 0$ for $n = 2$). We also verify the necessary injectivity radius lower bound for the smooth metrics in two parts. First, we analyze the injectivity radius near the cylindrical ends at infinity using the stability of the geodesic flow due to Sakai \cite{Sakai-1983-on-continuity-of-injectivity-radius-function}. Second, for a large enough compact set containing the medial axis of the domain, we bound the injectivity radius using the Fenchel--Borsuk theorem \cite{Fenchel,Borsuk} on the total curvature of $\mathcal{C}^2$-regular closed curves; we also employ Klingenberg's characterization of the injectivity radius on complete Riemannian manifolds. Our analysis leads to the $\mathrm{CAT}(2)$ property of the approximations and, by stability, for the domain itself.

In Section~\ref{theorem-proof-of-many-results}, we prove Theorems~\ref{thm:universal-cat2}, \ref{theorem-cbb-equivalent-to-concavity}, and \ref{thm:convex-cat0}. The proof of Theorem~\ref{thm:universal-cat2} essentially follows from the previous sections by stability. The proof of Theorem~\ref{theorem-cbb-equivalent-to-concavity} has two key directions. First, in case $\Omega$ is concave, the convexity and $\mathcal{C}^{1,1}_{\text{loc}}$-regularity of $\delta_\Omega$ follow from convex analysis. This allows us to establish a direct link between a distributional sectional curvature bound of the quasihyperbolic Riemannian metric $\delta_{\Omega}^{-2}g_{\euc}$ and the Alexandrov curvature lower bound of $( \Omega, k_\Omega )$ using the recent approximation results of $\mathcal{C}^1$-metrics from \cite{eros-kunzinger-obanyan-vardabasso-2026-distributional-sectional-curvature-bounds-for-riemannian-metrics-of-low-regularity}. In the converse direction, we employ the theory of metric spaces with two-sided Alexandrov curvature bounds established by Berestovski\u{\i} and Nikolaev \cite{Berestovskij-1976-introduction-of-a-riemann-structure-into-certain-metric-spaces,nikolaev-1983-smoothness-of-the-metric-of-spaces-with-bilaterally-bounded-curvature-in-the-sense-of-alexandrov} and the regularity theory of angle-preserving mappings for low-regularity Riemannian metrics by Iwaniec \cite{iwaniec-1982-regularity-theorems-for-solutions-of-partial-differential-equations-for-quasiconformal-mappings-in-several-dimensions}. Regarding the proof of Theorem~\ref{thm:convex-cat0}, we first establish the equivalence between the convexity of $\Omega$ and that of $x \mapsto -\log(\delta_{\Omega}(x))$. This links the proof to a recent conformal deformation theorem of $\mathrm{CAT}(0)$ spaces due to Lytchak and Stadler \cite{LytchakStadler}.

In Section~\ref{section-consequences-of-theorems}, we establish Theorem~\ref{theorem-resolution-of-conjectures} as a corollary of Theorem~\ref{thm:universal-cat2}. We also analyze families of quasihyperbolic geodesics for balls whose radii are below the injectivity radius $D_\kappa$, eventually leading to the proof of Theorem~\ref{theorem-c1-regular-spheres}. During the proof, we use the regularity of quasihyperbolic geodesics by Martin \cite{Martin1985-qc-and-bi-lipschitz-homeomorphisms-uniform-domains-and-the-quasihyperbolic-metric} and the variation estimates for geodesics in $\CAT(\kappa)$ spaces. Together with a Landau-type interpolation formula, the regularity of the geodesics lead to the $\mathcal{C}^{1,\frac{1}{2}}$-regularity of the metric spheres $0 < r < D_\kappa$. By using a radial projection argument due to Väisälä \cite{vaisala-2007-quasihyperbolic-geometry-of-domains-in-hilbert-spaces}, we prove that metric spheres with radius $0 < r < D_2$ are $\mathcal{C}^{1,\frac{1}{2}}$-diffeomorphic to the sphere $\Sph^{n-1}$. Extending the conclusion to the interval $0 < r < D_\kappa$ and deducing the existence of $\Psi$ from Theorem~\ref{theorem-c1-regular-spheres} goes back to a structure theorem of proper submersions due to Ehresmann~\cite{ehresmann-1951-infinitesimal-connections-in-differentiable-fiber-spaces}. However, a subtlety is caused by the low regularity of the gradient flow of $x \mapsto k_\Omega(p,x)$. We also prove the convexity result, Theorem~\ref{theorem-small-balls-euclidean-convex}, in this section.

In the final section, Section~\ref{section-examples}, we provide two examples. The first one establishes Theorem~\ref{thm:sharp-cat2}. The second example concerns a smooth concave non-convex domain which is $\mathrm{CAT}(0)$.

\section{Preliminaries on Riemannian geometry}\label{section-preliminaries-riemannian-geometry}

\subsection{Injectivity and conjugate radii}

Let $(M,g)$ be a connected complete Riemannian manifold (without boundary), and let $d_g$ be its Riemannian distance. By the Hopf--Rinow theorem, metric completeness is equivalent to geodesic completeness; hence the exponential map $\exp_p\colon T_pM\to M$ is defined on all of $T_pM$ for every $p\in M$.

\begin{definition}[Riemannian injectivity radius]
\label{def:riemannian-injectivity-radius}
The \emph{injectivity radius at $p$} is denoted by $\injrad_{g}(p)$ and defined as the supremum over the radii $r > 0$ for which $\left.\exp_p\right|_{B_{T_pM}(0,r)}$ is a diffeomorphism onto its image.

The \emph{global Riemannian injectivity radius} is
$$
\injrad_g(M)
\coloneqq
\inf_{p\in M}\injrad_g(p).
$$
The value $+\infty$ is allowed.
\end{definition}

For $v\in T_pM$ with $|v|_g=1$, put
$\gamma_v(t)=\exp_p(tv)$ and define its \emph{cut time} by
$$
c_p(v)
\coloneqq
\sup\left\{
t>0\colon
\gamma_v|_{[0,t]}
\text{ is minimizing}
\right\}
\in(0,+\infty].
$$
The \emph{cut locus} of $p$ is
$$
\Cut(p)
\coloneqq
\left\{
\exp_p(c_p(v)v)\colon
|v|_g=1,\ c_p(v)<\infty
\right\}.
$$
For a complete Riemannian manifold,
\begin{equation}
\injrad_g(p)
=
\inf_{|v|_g=1}c_p(v)
=
d_g\bigl(p,\Cut(p)\bigr),
\label{eq:injectivity-cut-locus}
\end{equation}
where the distance to the empty set is $+\infty$. The cut locus is
closed, and the function $p\mapsto\injrad_g(p)$ is continuous; see
\cite[Theorem~10.34(a) and Propositions~10.36--10.37]{Lee}.

\begin{definition}
Two points $p$ and $\gamma_v(t)$ are \emph{conjugate along}
$\gamma_v|_{[0,t]}$ if there is a nonzero Jacobi field along this
segment which vanishes at both endpoints. Equivalently, the differential
$(D\exp_p)_{tv}$ is not invertible. Define
$$
j_p(v)
\coloneqq
\inf\left\{
t>0\colon(D\exp_p)_{tv}\text{ is not invertible}
\right\},
$$
where $\inf\emptyset=+\infty$.
\end{definition}

\begin{definition}[Conjugate radius]
\label{def:conjugate-radius}
The conjugate radii at $p$ and of $(M,g)$ are
$$
\conjrad_g(p)
\coloneqq
\inf_{|v|_g=1}j_p(v),
\quad
\conjrad_g(M)
\coloneqq
\inf_{p\in M}\conjrad_g(p).
$$
\end{definition}
We conclude from \eqref{eq:injectivity-cut-locus} that
$$
\injrad_g(p)\leq\conjrad_g(p),
\quad
\injrad_g(M)\leq\conjrad_g(M).
$$
By \cite[Theorem~11.12]{Lee}, we have the following: if the sectional curvature of $g$ is at most $\kappa$, the Rauch
comparison theorem gives
    \begin{equation}\label{eq:rauch-conjugate-radius}
        \conjrad_g(M)
        \geq
        D_\kappa,
    \end{equation}
    where
    \begin{equation}\label{equation-d-kappa}
        D_\kappa
        =
        \begin{cases}
            \displaystyle\frac{\pi}{\sqrt\kappa},&\kappa>0,\\[6pt]
            +\infty,&\kappa\leq0.
        \end{cases}
    \end{equation} 
    
More generally, the following well-known result by Klingenberg relates the injectivity and conjugate radii.
\begin{theorem}[{Klingenberg's lemma, \cite[Lemma 1]{Klingenberg-1959-contributions-to-riemannian-geometry-in-the-large}}]\label{theorem-klingenberg-lemma}
    Let $(M,g)$ be a complete Riemannian manifold. In case $\injrad_g(p) < \conjrad_g(p)$, there exists a geodesic loop $\gamma \colon [0, \ell] \to (M,g)$ with $\gamma(0) = p = \gamma(\ell)$ and of length $\ell = 2 \injrad_g(p)$. If $\injrad_g( \gamma( \ell/2) ) = \ell/2$, then $\gamma$ extends to a smooth periodic geodesic $\R \to ( M, g )$.
\end{theorem}
In the statement of the theorem, we use the following definition.
\begin{definition}
For $L>0$, a unit-speed \emph{geodesic loop based at $p$} is a geodesic
$\gamma\colon[0,L]\to M$ satisfying
$\gamma(0)=\gamma(L)=p$. It is a \emph{smooth closed geodesic} if it
defines a smooth geodesic
$$
\R/(L\mathbb Z)\longrightarrow M,
$$
or, equivalently, if $\dot\gamma(0)=\dot\gamma(L)$.
\end{definition}
When $M=U\subset\R^n$ and $g$ is a smooth Riemannian metric on $M$, every nonconstant smooth
closed $g$-geodesic is a smooth regular closed curve in $\R^n$. We use this fact during the proof of Proposition~\ref{proposition-injectivity-radius-lower-bound}.

\subsection{Conformal changes and total curvature}\label{subsec:conformal-change}

Throughout the manuscript, $g_{\euc}$ denotes the Euclidean Riemannian metric.
For a function $f$, the symbols $\nabla f$ and $\Hess f$ denote its
Euclidean gradient and Hessian, respectively, while
$\langle\cdot,\cdot\rangle$ and $|\cdot|$ denote the Euclidean inner product and norm. Unless otherwise specified, orthogonality, orthogonal
complements, orthogonal projections, and traces are also understood
with respect to $g_{\euc}$.

For a Riemannian metric $g$, let $\nabla^g$ denote its Levi--Civita
connection. We use the curvature convention
\begin{equation*}
    R^g(X,Y)Z
    =
    \nabla_X^g\nabla_Y^gZ
    -\nabla_Y^g\nabla_X^gZ
    -\nabla_{[X,Y]}^gZ.
\end{equation*}
Thus, if $x\in U$ and $\Pi\subset T_xU$ is a two-plane, and
$(E_1,E_2)$ is a $g_x$-orthonormal basis of $\Pi$, then the sectional
curvature of $g$ at $x$ in the direction of $\Pi$ is
$$
K_g(x,\Pi)
=
g_x\bigl(R^g(E_1,E_2)E_2,E_1\bigr).
$$

For $u\in T_xU$, we use $u\otimes u$ to denote the symmetric covariant
bilinear form determined by the Euclidean metric:
$$
(u\otimes u)(v,z)
=
\langle u,v\rangle\langle u,z\rangle,
\quad v,z\in T_xU.
$$

Let $U\subset\R^n$ be open, let $F\in C^2(U)$, and set
$$
g=e^{2F}g_{\euc}.
$$
Below $D$ denotes the Euclidean Levi--Civita connection. The Koszul
formula gives
\begin{equation}
    \nabla_X^gY
    =
    D_XY+dF(X)Y+dF(Y)X
    -\langle X,Y\rangle\nabla F
    \label{eq:conformal-connection}
\end{equation}
for all $C^1$ vector fields $X,Y$ on $U$.

For later use, we record the following formula which is the special Euclidean version of the conformal transformation law for the Riemann curvature tensor; see
\cite[Theorem~7.30, equation~(7.44), p.~217]{Lee}.
We record it explicitly to fix notation.

\begin{lemma}[Conformal curvature formula]\label{lem:conformal-curvature}
Let $U\subset\R^n$ be open, $n\geq2$, $F\in C^2(U)$, and equip
$U$ with the Riemannian metric $g=e^{2F}g_{\euc}$. For $x\in U$ and a
two-plane $\Pi\subset T_xU$, one has
\begin{align}
e^{2F(x)}K_g(x,\Pi)
&=
-\operatorname{tr}_{\Pi}\bigl((\Hess F)_x\bigr)
+\left|\pi_{\Pi}\bigl(\nabla F(x)\bigr)\right|^2
-\left|\nabla F(x)\right|^2
\label{eq:general-conformal-curvature}\\
&=
-\operatorname{tr}_{\Pi}\bigl((\Hess F)_x\bigr)
-\left|
\pi_{\Pi^\perp}\bigl(\nabla F(x)\bigr)
\right|^2.
\nonumber
\end{align}
\end{lemma}
Here and below, $\pi_\Pi\colon T_xU\to\Pi$ is the Euclidean orthogonal projection,
$\Pi^\perp$ is the Euclidean orthogonal complement of $\Pi$, and
$$
\operatorname{tr}_{\Pi}\bigl((\Hess F)_x\bigr)
=
(\Hess F)_x(e_1,e_1)
+
(\Hess F)_x(e_2,e_2)
$$
for any Euclidean orthonormal basis $(e_1,e_2)$ of $\Pi$. Thus every
quantity on the right-hand side of
\eqref{eq:general-conformal-curvature} is computed using $g_{\euc}$,
whereas $K_g(x,\Pi)$ is computed using $g$.

We recall the classical Fenchel--Borsuk theorem \cite{Fenchel,Borsuk}.
\begin{theorem}[Fenchel--Borsuk]
	\label{thm:fenchel-borsuk}
	Let $n\geq2$ and let
	$\gamma\colon\R/(\mathcal T\mathbb Z)\to\R^n$ be a $C^2$
	regular closed curve, where $\mathcal T>0$. Reparametrize
	$\gamma$ by Euclidean arclength $s$, and let
	$$
	\mathbf t=\frac{d\gamma}{ds}
	$$
	be its Euclidean unit tangent vector. Then the total curvature of
	$\gamma$ satisfies
	\begin{equation*}
		\operatorname{TC}(\gamma)
		\coloneqq
		\int_\gamma
		\left|
		\frac{d\mathbf t}{ds}
		\right|\,ds
		\geq2\pi.
	\end{equation*}
\end{theorem}
Theorem~\ref{thm:fenchel-borsuk} is Fenchel's theorem in Borsuk's $n$-dimensional form; see \cite{Borsuk}.
We need the following consequence of
Theorem~\ref{thm:fenchel-borsuk}.

\begin{lemma}\label{lem:g-length-of-conformal-closed-geodesic}
Let $U\subset\R^n$ be open, let $F\in C^\infty(U)$, and consider $g=e^{2F}g_{\euc}$.
Assume that
\begin{equation*}
    |\nabla F|\leq e^F
    \quad\text{on }U.
\end{equation*}
Then every nonconstant smooth closed $g$-geodesic has $g$-length at
least $2\pi$.
\end{lemma}

\begin{proof}
Let $\gamma$ be such a geodesic. Reparametrize it
by Euclidean arclength, and let $T=d\gamma/ds$ be its Euclidean unit
tangent. While the new parameter need not be affine for $g$, the normal component of $\nabla_T^g T$ vanishes by the chain rule. In particular, there is a scalar function $\lambda$ such that
$$
\nabla_T^gT=\lambda T.
$$
Since
$g$ is conformal to $g_{\euc}$, the $g$-orthogonal and Euclidean-orthogonal
complements of $T$ agree. Taking this orthogonal component from \eqref{eq:conformal-connection} and using the identities $|T|=1$ and $\langle T',T\rangle=0$ gives
\begin{equation*}
    \lambda T = T' + 2\langle\nabla F,T\rangle T - \nabla F
    \Rightarrow
    0 = T' - ( \nabla F - \langle \nabla F, T \rangle T ).
\end{equation*}
Consequently,
$$
|T'|\leq|\nabla F| \leq e^F.
$$
By Theorem~\ref{thm:fenchel-borsuk}, we have
$$
2\pi
\leq\int_\gamma|T'| \,ds_0
\leq\int_\gamma e^F\,ds_0
=\ell_g(\gamma).
$$
The claim follows.
\end{proof}

\section{Preliminaries on metric geometry}\label{section-preliminaries-metric-geometry}

\subsection{Length of curves and length distances}\label{subsec:metric-geometry}
Let $(X,d)$ be a metric space. A {\it curve} $\gamma$ is a continuous function $\gamma\colon[a,b]\rightarrow X$. Let $\mathcal{P}$ denote the set of all partitions $a = t_{0}<t_{1}<t_{2}< \cdots<t_{n}=b$ of the interval $[a,b]$. The \emph{length} of the curve $\gamma$ in the metric space $(X,d)$ is denoted by $\ell_{d}(\gamma)$, and is defined as
$$\ell_{d}(\gamma) = \sup_{\mathcal{P}} \sum_{k=0}^{n-1}d(\gamma(t_{k}), \gamma(t_{k+1})).$$
A curve in $(X,d)$ is \emph{rectifiable} if $\ell_{d}(\gamma) < \infty$. A metric space $(X,d)$ is \emph{rectifiably connected} if every pair of points $x,y \in X$ can be joined by a rectifiable curve. 

We say that $(X,d)$ is a \emph{length space}, if $d(x,y)=\inf_{\gamma}\{\ell_d(\gamma)\}$ for all $x,y\in X$, where the infimum is taken over all rectifiable curves in $X$ joining $x$ and $y$. A length space is \emph{geodesic} if for every $x, y \in X$, there exists a curve $\gamma$ joining $x$ to $y$ with $d(x,y) = \ell_{d}(\gamma)$. In a locally compact length space, being complete is equivalent to properness by Hopf--Rinow theorem~\cite[Theorem 2.5.28]{BuragoBuragoIvanov}; note that a proper length space is geodesic.

Let $\gamma\colon[a,b]\to X$ be a rectifiable curve. Its arclength function
$s\colon[a,b]\to[0,\ell_d(\gamma)]$ is defined by $s(t)\coloneqq\ell_d\bigl(\gamma|_{[a,t]}\bigr)$.
	The arc length function is continuous and nondecreasing, with
	$s(a)=0$ and $s(b)=\ell_d(\gamma)$. Moreover, there is a
	unique $1$-Lipschitz map $\gamma_s\colon[0,\ell_d(\gamma)]\to X$ such that
	$\gamma=\gamma_s\circ s_\gamma$. The map $\gamma_s$ is called the \emph{arclength parametrization} of $\gamma$. We say that a non-constant rectifiable curve $\gamma \colon [a,b] \to X$ has \emph{constant-speed} if $\gamma_s(t) = \gamma( a + t\frac{b-a}{\ell_{d}(\gamma)} )$ for each $0 \leq t \leq \ell_{d}(\gamma)$.

\subsection{Pointed convergence}
\label{subsec:pointed-GH}

We recall pointed Gromov--Hausdorff convergence in this subsection.
For a map $f\colon(X,d_X)\to(Y,d_Y)$, its distortion is
$$
\operatorname{dis}(f)
=
\sup_{x,x'\in X}
\left|
d_Y\bigl(f(x),f(x')\bigr)-d_X(x,x')
\right|.
$$
If $A\subset Y$ and $\varepsilon>0$, let
$$
N_\varepsilon^Y(A)
=
\{y\in Y\colon d_Y(y,A)<\varepsilon\}.
$$

We use the following formulation of pointed Gromov--Hausdorff
convergence for proper spaces; it is equivalent to
\cite[Definition~8.1.1]{BuragoBuragoIvanov}.

\begin{definition}[Pointed Gromov--Hausdorff convergence]
\label{def:pointed-GH}
Let $(X_j,d_j,o_j)$ and $(X,d,o)$ be pointed proper metric spaces. We
write
$$
(X_j,d_j,o_j)
\xrightarrow{\mathrm{pGH}}
(X,d,o)
$$
if, for every $R>0$ and every $\varepsilon\in(0,R)$, there is $j_0$
such that, for every $j\geq j_0$, there is a map $f_j:B_d(o,R)\to X_j$ such that $f_j(o)=o_j$, $\operatorname{dis}(f_j)<\varepsilon$, and
$$
B_{d_j}(o_j,R-\varepsilon)
\subset
N_\varepsilon^{X_j} \left(f_j\bigl(B_{d}(o,R)\bigr)\right).
$$
The maps $f_j$ are not required to be continuous.
\end{definition}

\subsection{Alexandrov geometry} For $\kappa\in\R$, let $\mathbb M_\kappa^2$ be the complete simply
connected two-dimensional Riemannian manifold of constant sectional curvature $\kappa$, and denote its Riemannian distance and diameter by $d_\kappa$ and $D_\kappa$, respectively. In fact, $D_\kappa$ coincides with the definition in \eqref{equation-d-kappa}.

A metric space is called \emph{$D_\kappa$-geodesic} if every pair of points
at distance less than $D_\kappa$ can be joined by a geodesic segment.

A \emph{geodesic triangle} $\Delta(x,y,z)$ is a union of three geodesics joining $x$ to $y$, $y$ to $z$, and $z$ to $x$. The \emph{perimeter} is the sum of these lengths.

Let $\Delta(x,y,z)$ be a geodesic triangle of perimeter less than
$2D_\kappa$. A \emph{comparison triangle}
$$
\overline\Delta(\bar x,\bar y,\bar z)
\subset\mathbb M_\kappa^2
$$
is a geodesic triangle together with a \emph{comparison map} $\bar u \colon \overline{\Delta} \to \Delta$ sending $\bar{x}$ to $x$, $\bar{y}$ to $y$, $\bar{z}$ to $z$, and the corresponding sides to sides isometrically.

\begin{definition}[Upper Alexandrov curvature bound]
\label{def:CAT-kappa}
A metric space $(X,d)$ is a \emph{$\CAT(\kappa)$ space}
if it is $D_\kappa$-geodesic and the comparison map is $1$-Lipschitz for every geodesic triangle $\Delta$ of perimeter less than $2D_\kappa$. 

A metric space $(X,d)$ is \emph{locally $\CAT(\kappa)$} if every point has an open neighborhood which, equipped with the restricted
metric, is a $\CAT(\kappa)$ space.
\end{definition}
We refer to local $\CAT(\kappa)$ also as \emph{Alexandrov curvature upper bound $\leq \kappa$}. For the definition of lower bounds, see e.g. \cite[Chapter 8]{AlexanderKapovitchPetrunin}.

Below we need a local-to-global property for the curvature upper bounds.

\begin{theorem}[{Patchwork globalization, \cite[Theorem~9.30]{AlexanderKapovitchPetrunin}}]\label{thm:patchwork-globalization}
Let $(X,d)$ be a complete length space. Then the following conditions
are equivalent:
\begin{enumerate}
\item $(X,d)$ is $\CAT(\kappa)$;
\item $(X,d)$ is locally $\CAT(\kappa)$, every pair of points at
distance less than $D_\kappa$ is joined by a unique geodesic, and
these geodesics depend continuously on their endpoints.
\end{enumerate}
\end{theorem}
\begin{remark}\label{rem:continuous-dependence}
    Here continuous dependence means that, after parametrizing the unique geodesics
by constant speed (i.e., proportionally to arclength) on $[0,1]$, the map $(x,y,t)\longmapsto\gamma_{x,y}(t)$ is continuous on
$$
\{(x,y)\in X\times X\colon d(x,y)<D_\kappa\}\times[0,1].
$$
\end{remark}

We use the following corollary of Theorem~\ref{thm:patchwork-globalization}.
\begin{corollary}[Globalization by short-geodesic uniqueness]\label{cor:globalization-by-uniqueness}
Let $(X,d)$ be a proper length space which is locally
$\CAT(\kappa)$. Suppose that every pair of points at distance less
than $D_\kappa$ is joined by a unique geodesic. Then $(X,d)$ is
$\CAT(\kappa)$.
\end{corollary}

\begin{proof}
A proper length space is complete and geodesic. The required continuous dependence of geodesics follows from properness and Arzelà--Ascoli theorem.
\end{proof}

We also recall that on (smooth) Riemannian manifolds, the Alexandrov curvature upper bounds are coincide with sectional curvature upper bounds.
\begin{theorem}[{\cite[Theorem~II.1A.6]{BridsonHaefliger}}]\label{thm:riemannian-local-CAT}
Let $(M,g)$ be a connected Riemannian manifold and let $d_g$
be its Riemannian distance. Then
$$
K_g(x,\Pi)\leq\kappa
$$
for every $x\in M$ and every two-plane $\Pi\subset T_xM$ if and only
if $(M,d_g)$ is locally $\CAT(\kappa)$.
\end{theorem}

Note that no completeness
assumption is needed for this local equivalence. Combining this with the patchwork globalization, we have the following.
\begin{corollary}[Riemannian globalization criterion]\label{cor:riemannian-globalization}
 Let $\kappa>0$ and let $(M,g)$ be a connected Riemannian
manifold. If
$$
K_g\leq\kappa
\quad\text{and}\quad
\injrad_g(M)\geq D_\kappa,
$$
then $(M,d_g)$ is $\CAT(\kappa)$.   
\end{corollary}

\begin{proof}
Theorem~\ref{thm:riemannian-local-CAT} gives the local
$\CAT(\kappa)$ property. The Hopf--Rinow theorem shows that $(M,d_g)$
is proper, and \cite[Proposition~6.11, pp.~161--162, and Problem~6-16(a), p.~188]{Lee} gives the
required uniqueness for pairs at distance less than $D_\kappa$.
Corollary~\ref{cor:globalization-by-uniqueness} completes the proof.
\end{proof}

When $\kappa \leq 0$, the famous Cartan--Hadamard theorem gives a stronger equivalence.
\begin{theorem}[{\cite[Chapter~II, Theorem~4.1]{BridsonHaefliger} }]\label{theorem-cartan-hadamard}
    Let $(X,d)$ be a complete length space and $\kappa \leq 0$. Then the following conditions
    are equivalent:
    \begin{enumerate}
        \item $(X,d)$ is $\CAT(\kappa)$;
        \item $(X,d)$ is locally $\CAT(\kappa)$ and simply connected.
\end{enumerate}
\end{theorem}

We next recall the stability of the $\CAT(\kappa)$ property under pointed Gromov-Hausdorff convergence.
\begin{theorem}[{\cite[Chapter~II, Corollary~3.10(1)]{BridsonHaefliger}}]\label{thm:CAT-stability}
Let $\kappa \in \R$ and let $(X_j,d_j,o_j)$ be pointed $\CAT(\kappa)$ spaces.
If
$$
(X_j,d_j,o_j)
\xrightarrow{\mathrm{pGH}}
(X,d,o),
$$
then $(X,d)$ is $\CAT(\kappa)$.
\end{theorem}

\subsection{Basic estimates on the quasihyperbolic metric}\label{subsection-basic-estimates-qh-metric}
Let $\Omega \subsetneq \R^n$ be a domain. The quasihyperbolic metric is defined by
\begin{align*}
    k_{\Omega}(x,y) = \inf\{ \ell_{k_\Omega}(\gamma) \colon \text{rectifiable $\gamma\colon x\curvearrowright y$ with $\gamma \subset \Omega$} \}.
\end{align*}
We recall the following basic estimates for the quasihyperbolic metric established by Gehring and Palka \cite[Lemma~2.1]{Gehring-Palka-1976} in $\mathbb{R}^n$. For all $x,y\in\Omega$, and any rectifiable curve $\gamma$ in $\Omega$ joining $x$ and $y$,
	\begin{align}\label{eq:qh-estimate-1}
		k_{\Omega}(x, y) &\ge \log\left(1 + \frac{|x - y|}{\min\{\delta_{\Omega}(x), \delta_{\Omega}(y)\}}\right) 
		\ge \left| \log \frac{\delta_{\Omega}(y)}{\delta_{\Omega}(x)} \right|;
	\end{align}
here $\delta_{\Omega}(x) = \dist(\R^n \setminus \Omega,x)$ is the distance from $x$ to the complement of $\Omega$.
The first inequality above is a special case of the more general inequality

\begin{equation}\label{eq:qh-estimate-2}
     \ell_{k_{\Omega}}(\gamma)
    \geq
    \log\left(
        1 + \frac{ \ell( \gamma ) }{ \dist(|\gamma|, \partial \Omega ) }
    \right).
\end{equation}

We recall a well-known approximation result. We present a proof in order to be self-contained. Suppose $\Omega \subset \mathbb{R}^n$, $n\ge 2$, is a domain. Let $K = \mathbb{R}^n \setminus \Omega$. Let $P_1 \subset \dots \subset P_{i} \subset P_{i+1} \subset \dots \subset K$ be finite sets such that $\bigcup P_i$ is dense in $K$. Let also $\Omega_i \coloneqq \mathbb{R}^n \setminus P_{i}$. Notice that $\Omega\subset \Omega_{i+1}\subset \Omega_i$. 

\begin{lemma}\label{lemma-monotonicity}
For every $i\ge 1$, $\delta_{\Omega_i}\ge \delta_{\Omega_{i+1}}\ge \delta_{\Omega}$ on $\mathbb{R}^n$, and $\delta_{\Omega_i}\to \delta_{\Omega}$ locally uniformly on $\mathbb{R}^n$. Consequently, $\frac{ 1 }{ \delta_{\Omega_i} } \leq \frac{1}{\delta_{\Omega}}$ pointwise in $\Omega$. Moreover, $\frac{ 1 }{ \delta_{\Omega_i} } \to \frac{ 1 }{ \delta_{\Omega} }$ uniformly in compact subsets of $\Omega$.
\end{lemma}

\begin{proof}
The inclusion $P_i\subset P_{i+1}\subset K$ gives $\delta_{\Omega_i}\ge \delta_{\Omega_{i+1}}\ge \delta_{\Omega}$. Furthermore, by the definition of infimum and since $\bigcup P_i$ is dense in $K$, for every $x\in \mathbb{R}^n$, we have $\lim_{i\to \infty} \delta_{\Omega_i}(x)=\delta_{\Omega}(x)$. Since all the functions involved are continuous, Dini's theorem gives uniform convergence on every compact subset of $\mathbb{R}^n$. Note that Dini's theorem applies in compact subsets of $\Omega$ both to $( \delta_{\Omega_i} )_i$ and to $( 1/\delta_{\Omega_i} )_i.$
\end{proof}

\begin{proposition}\label{proposition-convergence-of-approximations}
If $\Omega$ and $( \Omega_i )_i$ are as above and $o \in \Omega$, then 
$$( \Omega_i, k_{\Omega_i}, o) \xrightarrow{\mathrm{pGH}} ( \Omega, k_{\Omega}, o ).$$
\end{proposition}
Proposition~\ref{proposition-convergence-of-approximations} gives the first approximation of a quasihyperbolic domain we use in this manuscript. The next section concerns another level of approximation.
\begin{proof}
We use the following consequence of \eqref{eq:qh-estimate-2},
\begin{align}\label{equation-curve-distance-estimates-quantitative}
    ( e^{ \ell_{k_\Omega}(\gamma) } - 1 ) \dist( |\gamma|, \partial \Omega) \geq \ell(\gamma) 
\end{align}
during the proof.

Suppose $o \in \Omega$ and $r > 0$. The identity map $f_i \colon \overline{B}_{k_{\Omega}}(o,r) \to \overline{B}_{k_{\Omega_i}}(o,r)$ is well-defined and $\operatorname{dis}(f_i) \rightarrow 0$ as $i \to \infty$. Indeed, by Lemma \ref{lemma-monotonicity}, we obtain that each $f_i$ is $1$-Lipschitz so
\begin{align}\label{equation-closed-balls-inclusion}
    \overline{B}_{k_\Omega}(o,r) \subset \overline{B}_{k_{\Omega_i}}(o,r).
\end{align}
Next, we claim that the Hausdorff distance between these balls converge to zero. To this end, let $x_i \in \overline{B}_{k_{\Omega_i}}(o,r)$. Consider a quasihyperbolic geodesic $\gamma_i$ joining $x_i \in \overline{B}_{k_{\Omega_i}}(o,r)$ to $o$. Then it holds that
\begin{align}\label{equation-curve-distance-estimates}
    ( e^{r} - 1 )\delta_{\Omega_i}( o ) 
    &\geq ( e^{r}-1 ) \dist( |\gamma_i|, \partial \Omega_{i}) 
    \\ \notag
    &\geq ( e^{k_{\Omega_i}(o,x_i)}-1 )  \dist( |\gamma_i|, \partial \Omega_{i})  \geq \ell( \gamma_i )
\end{align}
by the above. If we parametrize $\gamma_i \colon [0,1] \to \Omega_i$ with constant-speed with respect to the Euclidean length, we observe that the sequence is uniformly bounded and equicontinuous. Indeed, here $\lim_{i\to\infty} \delta_{\Omega_i}(o) = \delta_{\Omega}(o) > 0$ so \eqref{equation-curve-distance-estimates} and Arzelà--Ascoli theorem establish the claim. Therefore, any subsequential limit of $( \gamma_i )_i$ is a Lipschitz curve $\gamma \colon [0,1] \to \Omega$. Then \eqref{equation-curve-distance-estimates-quantitative}, \eqref{equation-curve-distance-estimates}, and Lemma~\ref{lemma-monotonicity} imply that
\begin{align*}
    k_{\Omega}(o, \gamma(1))\leq \ell_{k_\Omega}(\gamma) = \lim_{ k \to \infty } k_{\Omega_{i_k}}(o,x_{i_k}) \leq r.
\end{align*}
Consequently, any accumulation point of $( x_i )_i$ is contained in $\overline{B}_{k_\Omega}(o,r)$. The claimed Hausdorff convergence follows from the arbitrariness of $( x_i )_i$.

We prove next that 
\begin{align}\label{eq-distortion-convergence}
    \lim_{ i \to \infty } \operatorname{dis}(f_i) = 0.
\end{align}
As an initial step, we prove for every pair $x, y \in \overline{B}_{k_{\Omega}}(o,r)$ that
\begin{align*}
    \lim_{i \to \infty} \left| k_{\Omega_i}(x,y) - k_{\Omega}(x,y) \right| = 0.
\end{align*}
Let $r = k_{\Omega}(x,y)$ and $r_i = k_{\Omega_i}(x,y)$, noting $r_i \leq r$ by the $1$-Lipschitz property of $f_i$. By the Hausdorff convergence above, $y \in \overline{B}_{k_{\Omega}}( x, \liminf_{i} r_i )$. Therefore $r\leq \liminf_{i} r_i \leq \limsup_{i} r_i \leq r$. Hence $r = \lim_{i} r_i$ shows the claim.

Finally, let $x_i,y_i \in \overline{B}_{k_\Omega}(o,r)$ satisfy
\begin{align*}
    \operatorname{dis}(f_i) = k_{\Omega}(x_i,y_i) - k_{\Omega_i}(x_i,y_i).
\end{align*}
We pass to a subsequence for which $\operatorname{dis}(f_{i_k}) \to \limsup_{i} \operatorname{dis}(f_i)$. We may also suppose that $(x,y) = \lim_{k\to \infty}(x_{i_k},y_{i_k})$. By triangle inequality and the $1$-Lipschitz property of $f_i$, we conclude that
\begin{align*}
    \lim_{k\to\infty}  k_{\Omega}(x_{i_k},y_{i_k}) - k_{\Omega_{i_k}}(x_{i_k},y_{i_k}) 
    = 
    \lim_{k\to\infty}  k_{\Omega}(x,y) - k_{\Omega_{i_k}}(x,y) 
    = 0.
\end{align*}
The claim follows. The convergence in Hausdorff distance and \eqref{eq-distortion-convergence} imply the required pointed Gromov--Hausdorff convergence by the arbitrariness of $r >0$.
\end{proof}

\section{Quasihyperbolic metric in finitely connected domains}\label{section-qh-metric-in-finite-punctures}
    The goal of this section is to prove the following theorem.

   \begin{theorem}\label{thm:finite-punctures-CAT-2}
   If $\Omega = \R^n \setminus P$ for a finite set $P \neq \emptyset$, the quasihyperbolic domain $(\Omega,k_\Omega)$ is $\CAT(2)$ if $n \geq 3$ and $\CAT(1)$ if $n = 2$. Moreover, if $n = 2$, $p \in \Omega$, and $0 < r \leq \pi/2$, then $B_{k_\Omega}(p,r)$ is $\CAT(0)$.
   \end{theorem}
   The $n = 2$ case is contained in \cite{Herron-2020,HerronMartinSubharmonic} but our methods lead to a different proof of this fact.
    
    We fix the following notation for this section. We consider a finite set $\emptyset \neq P = \{p_1,\dots,p_m\} \subset \R^n$ and $\Omega = \R^n \setminus P$ for $n \geq 2$. 
    We denote\begin{equation*}
	r_i(x)=|x-p_i|,
	\quad
	\delta_P(x)=\min_i r_i(x),
	\quad
	\rho_P(x)=\frac1{\delta_P(x)}
	=\max_i\frac1{r_i(x)}.
\end{equation*}
    We also denote the quasihyperbolic metric on $\Omega$ by $g_{P} = \rho_{P}^{2}g_{\euc}$.

\subsection{The smooth approximation}
    For each $\beta \in [1,\infty)$, we define
    \begin{align*}
        \rho_\beta(x) = \left( \sum_{i=1}^{m} \left( \frac{1}{r_i(x)} \right)^{\beta} \right)^{ \frac{1}{\beta} }.
    \end{align*}
    The definition is motivated by the following proposition.
\begin{proposition}\label{proposition-smooth-approximation}
    The Riemannian metrics $g_\beta = \rho_{\beta}^{2}g_{\euc}$ define a smooth approximation of $\rho_{P}^{2}g_{\euc}$ as $\beta \rightarrow \infty$. More precisely, the metric $g_\beta$ is smooth,
    \begin{align*}
        \rho_P \leq \rho_{\beta} \leq m^{1/\beta} \rho_P,
    \end{align*}
    and $\lim_{\beta \rightarrow \infty} \rho_{\beta} = \rho_P$ uniformly in compact subsets of $\Omega$. In particular, $(\Omega,d_\beta)$ is complete and proper. Moreover, for every sequence $\beta_j\to \infty$ and 
every base point $o\in \Omega$,
\begin{equation*}
    (\Omega,d_{\beta_j},o)
    \xrightarrow{\mathrm{pGH}}(\Omega,k_{\Omega},o) \quad \text{ as } j\to \infty
\end{equation*}
with the convergence realized by the identity map.
\end{proposition}
\begin{proof}
    The elementary $\ell^{\beta}$-norm inequality gives $\rho_{P} \leq \rho_{\beta} \leq m^{1/\beta} \rho_P$ and hence, after integrating along curves and taking infima gives $k_\Omega \leq d_{\beta} \leq m^{1/\beta} k_\Omega$. Therefore, for every sequence $\beta_j\to\infty$ and every basepoint $o\in \Omega$, the pointed Gromov--Hausdorff convergence follows with a similar argument to that of Proposition~\ref{proposition-convergence-of-approximations}.
\end{proof}

\subsection{The geometry of the smooth approximation}
Motivated by the approximation result, Proposition \ref{proposition-smooth-approximation}, we prove the following theorem.
\begin{theorem}\label{theorem-smooth-approximation}
    For every $\beta \in [1,\infty)$, the space $( \Omega, d_{\beta})$ is $\CAT(2)$ if $n \geq 3$ and $\CAT(1)$ if $n = 2$.  Moreover, if $n = 2$, $p \in \Omega$, and $0 < r \leq \pi/2$, then $B_{k_\Omega}(p,r)$ is $\CAT(0)$.
\end{theorem}

Theorem~\ref{theorem-smooth-approximation} is proved in two parts.
First, we have the following conclusion on the sectional curvature.
\begin{proposition}\label{proposition-sectional-curvature-upper-bound}
    For every $\beta \in [1,\infty)$, the sectional curvature of $g_\beta$ is bounded from above by $2$. More precisely, for every $\beta\geq1$, every $x\in \Omega$, and every
two-dimensional subspace $\Pi\subset T_x\Omega$, one has
    \begin{equation*}
        K_{g_\beta}(x,\Pi)
        \leq
        \begin{cases}
	        0, & \text{if $n=2$},\\
	        2, & \text{if $n\geq 3$}.
        \end{cases}
    \end{equation*}
\end{proposition}
    Second, we need a lower bound on the injectivity radius.
\begin{proposition}\label{proposition-injectivity-radius-lower-bound}
    For every $\beta \in [1,\infty)$, 
    \begin{equation}\label{eq:global-inj}
        \inf_x \operatorname{injrad}_{g_\beta}(x)
        \geq
        \begin{cases}
	        \pi, & \text{if $n=2$},\\
	        \pi/\sqrt{2}, & \text{if $n\geq 3$}.
        \end{cases}
    \end{equation}
\end{proposition}

\begin{proof}[Proof of Theorem \ref{theorem-smooth-approximation}]
    Let $n \geq 3$. By Proposition~\ref{proposition-sectional-curvature-upper-bound} and Theorem \ref{thm:riemannian-local-CAT}, $(\Omega, d_\beta)$ is locally $\CAT(2)$. By Proposition~\ref{proposition-injectivity-radius-lower-bound}, if
    $d_\beta(x,y)<D_2$, then $y$ lies inside the injectivity-radius ball at
    $x$, so $x$ and $y$ are joined by a unique minimizing geodesic. Also, $(\Omega, d_{\beta})$ is a proper complete length space by Proposition~\ref{proposition-smooth-approximation}. Hence, by Corollary~\ref{cor:globalization-by-uniqueness}, $(\Omega, d_{\beta})$ is $\CAT(2)$. The argument for the $\CAT(1)$ property when $n = 2$ is similar. To finish the argument for $n = 2$, we see that $(\Omega, d_\beta )$ is locally $\CAT(0)$ which implies that the ball $B_{d_\beta}(p,r)$, for $0< r \leq \pi/2$, is complete and locally $\CAT(0)$. By the $\CAT(1)$ property, the ball is also convex and contractible along the radial geodesics. Thus, by the Cartan--Hadamard theorem, the ball is $\CAT(0)$ as claimed.
\end{proof}

     For the proofs of the above propositions, we need the following auxiliary functions:

\begin{align*}
    F_\beta=\log\rho_\beta=
\frac1\beta
\log\left(\sum_{i=1}^m e^{\beta f_i}\right)
	\quad\text{and}\quad
	g_\beta=\rho_\beta^2g_{\euc}.
\end{align*}
Let also $ f_i=-\log r_i$,
$$
w_{i,\beta}
=
\frac{e^{\beta f_i}}
     {\sum_{j=1}^m e^{\beta f_j}}
=
\frac{r_i^{-\beta}}
     {\sum_{j=1}^m r_j^{-\beta}},
\quad
b_i=\nabla f_i,
\quad
b_\beta=\nabla F_\beta.
$$
Thus $w_{i,\beta}>0$ and
$$
\sum_{i=1}^m w_{i,\beta}=1.
$$
Finally, define the covariance tensor
$$
\mathcal C_\beta
=
\sum_{i=1}^m w_{i,\beta}b_i\otimes b_i
-
b_\beta\otimes b_\beta.
$$

\begin{proof}[Proof of Proposition~\ref{proposition-sectional-curvature-upper-bound}]
Differentiating the expression for $F_\beta$ gives
\begin{equation*}
    b_\beta
    =
    \sum_{i=1}^m w_{i,\beta}b_i
    \quad\text{and}\quad
    \Hess F_\beta
    =
    \sum_{i=1}^m w_{i,\beta}\Hess f_i
    +
    \beta\mathcal C_\beta.
\end{equation*}
Moreover, for every $v\in T_x\Omega$,
$$
\begin{aligned}
\mathcal C_\beta(v,v)
&=
\sum_{i=1}^m
w_{i,\beta}\langle b_i,v\rangle^2
-
\left(
\sum_{i=1}^m
w_{i,\beta}\langle b_i,v\rangle
\right)^2 \\
&=
\sum_{i=1}^m
w_{i,\beta}
\langle b_i-b_\beta,v\rangle^2
\geq0.
\end{aligned}
$$
Thus $\mathcal C_\beta$ is positive semidefinite. For each $i$, direct differentiation gives
\begin{equation}
b_i(x)
=
-\frac{x-p_i}{r_i(x)^2},
\quad
|b_i(x)|=\frac1{r_i(x)},
\quad
\Hess f_i
=
-\frac1{r_i^2}g_{\euc}+2b_i\otimes b_i.
\label{eq:fi-derivatives}
\end{equation}

Fix a two-dimensional subspace $\Pi\subset T_x\Omega$. By
Lemma~\ref{lem:conformal-curvature}, and since
$g_\beta=e^{2F_\beta}g_{\euc}=\rho_\beta^2g_{\euc}$, we have
$$
\rho_\beta^2K_{g_\beta}(x,\Pi)
=
-\operatorname{tr}_\Pi(\Hess F_\beta)
+
|\pi_\Pi b_\beta|^2
-
|b_\beta|^2.
$$
Equation~\eqref{eq:fi-derivatives} yields
$$
\begin{aligned}
-\operatorname{tr}_\Pi(\Hess f_i)
&=
\frac2{r_i^2}
-
2|\pi_\Pi b_i|^2 
=
2|\pi_{\Pi^\perp} b_i|^2.
\end{aligned}
$$
Furthermore,
$$
\operatorname{tr}_\Pi\mathcal C_\beta
=
\sum_{i=1}^m
w_{i,\beta}|\pi_\Pi b_i|^2
-
|\pi_\Pi b_\beta|^2.
$$
Substitution therefore gives the exact identity
\begin{equation}
\begin{aligned}
\rho_\beta^2K_{g_\beta}(x,\Pi)
={}&
2\sum_{i=1}^m
w_{i,\beta}|\pi_{\Pi^\perp}b_i|^2
-
|\pi_{\Pi^\perp}b_\beta|^2
-
\beta\operatorname{tr}_\Pi\mathcal C_\beta.
\end{aligned}
\label{eq:exact-curvature}
\end{equation}
Since $\mathcal C_\beta$ is positive semidefinite,
\eqref{eq:exact-curvature} implies
$$
\begin{aligned}
\rho_\beta^2K_{g_\beta}(x,\Pi)
&\leq
2\sum_{i=1}^m w_{i,\beta}|b_i|^2
=
2\sum_{i=1}^m\frac{w_{i,\beta}}{r_i^2}
\leq
2\rho_\beta^2.
\end{aligned}
$$
The last inequality follows from
$r_i^{-1}\leq\rho_\beta$ and
$\sum_iw_{i,\beta}=1$. Dividing by $\rho_\beta^2$ proves the
claim when $n \geq 3$. When $n = 2$, it holds that $\pi_{\Pi^{\perp}}$ is the zero map so $\rho_{\beta}^{2}K_{g_\beta}(x,\Pi) = -\beta\operatorname{tr}_\Pi\mathcal C_\beta \leq 0$. The claim follows.
\end{proof}

\begin{remark}[The covariance contribution]
\label{rem:covariance-contribution}
The $\beta$-weighted covariance term in
\eqref{eq:exact-curvature} satisfies
$$
-\beta\operatorname{tr}_\Pi\mathcal C_\beta
=
-\beta
\sum_{i=1}^m
w_{i,\beta}
|\pi_\Pi(b_i-b_\beta)|^2
\leq0.
$$
Thus the covariance term arising from the variation of the
log-sum-exp weights cannot make a positive contribution to the
sectional curvature. In particular, the upper estimate remains
uniform at the medial axis. This does not mean that the total sectional curvature is
nonpositive there: the remaining terms in
\eqref{eq:exact-curvature} may be positive and can reach the
upper bound $2$ when $n \geq 3$. We discuss an example to this effect in Section \ref{section-sharpness-of-the-curvature-bound}.
\end{remark}

For the proof of Proposition~\ref{proposition-injectivity-radius-lower-bound}, we need to analyze the injectivity radius in two regions: near the punctures $P$ or the infinity point $\{\infty\}$ and near the boundaries of the Voronoi cells associated to $P$. We start with the following lemma.

\begin{lemma}[Cylindrical end control]\label{lem:ends}
For every $\beta \in [1,\infty)$ and $\varepsilon > 0$, there is a compact set
$K_\beta\subset \Omega$ such that
\begin{equation*}
	\inj_{g_{\beta}}(x)>\pi-\varepsilon
	\quad\text{for }x\in \Omega\setminus K_\beta.
\end{equation*}
\end{lemma}
\begin{proof}
For the duration of the proof, we fix $\beta \in [1,\infty)$. It suffices to prove that if $( x_k )_k$ is a sequence in $\R^n \setminus P$ converging to a point $p_j \in P$ or to $p_\infty = \infty$, then $\liminf_{k} \inj_{g_{\beta}}(x_k) \geq \pi$.

We first consider the case where $x_k \rightarrow p_j$. We argue by contradiction: $r_0 = \liminf_{ k } \inj_{g_{\beta}}( x_k ) < \pi$. Letting $\varepsilon > 0$ be small enough, by passing to a subsequence and relabeling, we may suppose that 
\begin{equation*}
    \inj_{g_{\beta}}( x_k ) < \frac{ \pi }{ \sqrt{1+\varepsilon} }
\end{equation*}
for each $k \in \mathbb{N}$. 

We derive a contradiction by considering suitable approximation at the end $p_j$ at infinity. For this, let $\delta = \min_{ a \neq b } |p_a-p_b|$. Up to passing to a tail of the sequence, we may suppose that $\sup_{k} |x_k-p_j| < \delta/3$. For $s>0$, consider
$$
	\Phi_{j,s}(u,\theta)=p_j+e^{-(s+u)}\theta,
	\quad
	(u,\theta)\in U_s
$$
for the domain $U_{s} = \{ (u,\theta) \in \R \times \Sph^{n-1} \colon u > -s - \log(\delta/3) \}$. By construction,
\begin{align*}
    \Phi_{j,s}^{*}g_{\euc} = e^{-2(s+u)}( du^2 + g_{ \Sph^{n-1} } ).
\end{align*}
Consequently,
\begin{equation*}
	\Phi_{j,s}^*g_\beta
	=
	A_{j,s}(u,\theta)^2
	\bigl(du^2+g_{\Sph^{n-1}}\bigr),
\end{equation*}
where
\begin{equation*}
	A_{j,s}^{\beta}
	=
	1+
	\sum_{i\neq j}
	\left(
	\frac{e^{-(s+u)}}
	{|p_j-p_i+e^{-(s+u)}\theta|}
	\right)^\beta.
\end{equation*}
Indeed,
$A_{j,s}=e^{-(s+u)}\rho_\beta\circ\Phi_{j,s}$.
In particular, in $U_s$, it holds that
\begin{align*}
    |p_j-p_i+e^{-(s+u)}\theta| \geq \delta/3.
\end{align*}
It follows that
\begin{equation*}
	A_{j,s}\longrightarrow1
	\quad\text{in }C^\infty_{\mathrm{loc}}
		(\R\times\Sph^{n-1})
	\quad\text{as }s\to\infty.
\end{equation*}
Here $C^\infty_{\mathrm{loc}}$ convergence means that the functions and each of their iterated derivatives converge uniformly on every compact subset. For maps that are eventually defined on each compact subset, as above, this statement is understood for all sufficiently large $s$ on that subset.

By construction, there are parameters $( s_k )_k$ and $( \theta_k )_k$ in $\Sph^{n-1}$ so that $\Phi_{j,s_k}( 0, \theta_k ) = x_k$ for each $k \in \mathbb{N}$ and $s_k \rightarrow \infty$. 

The metrics $\Phi^{*}_{j,s_k}g_\beta$ are defined on the variable domains $U_{s_k} = (-a_k,\infty) \times \Sph^{n-1}$ for $a_k \to \infty$, so we lose no generality in supposing that $-a_k < -2\pi$ for every $k$. Now, by considering a smooth cut-off function depending on $a_k$, we find a smooth Riemannian metric $g_k$ on $\R \times \Sph^{n-1}$ that agrees with the product metric $g_{ \R \times \Sph^{n-1} }$ in $(-\infty,-a_k) \times \Sph^{n-1}$ and with the pullback $\Phi^{*}_{j,s_k}g_\beta$ on $[-2\pi,\infty) \times \Sph^{n-1}$. We may arrange things so that $( g_k )_k$ converges to $g_{ \R \times \Sph^{n-1} }$ in $C^\infty_{\mathrm{loc}} (\R\times\Sph^{n-1})$. By the smooth convergence, the sectional curvatures of these metrics converge uniformly in compact sets. Since $\R \times \mathbb{S}^{n-1}$ has sectional curvature at most one, for every sufficiently large $k$, the sectional curvatures of $g_k$ are bounded from above by $1+\varepsilon$ in $[-2\pi,2\pi] \times \Sph^{n-1}$. 

By the smooth convergence of the metrics, for every sufficiently large $k$,
\begin{align*}
    \frac{ 1 }{ 1+\varepsilon } g_k \leq g_{ \R \times \Sph^{n-1} } \leq (1+\varepsilon)g_k
\end{align*}
in $T( (-2\pi,2\pi) \times \Sph^{n-1} )$. Moreover, by Rauch comparison theorem, see~\eqref{eq:rauch-conjugate-radius}, a $g_k$-geodesic starting at $\{0\} \times \Sph^{n-1}$ and of $g_k$-length $<\pi/\sqrt{1+\varepsilon}$ contains no conjugate points. It now follows by Theorem~\ref{theorem-klingenberg-lemma} that there exists a closed $g_k$-geodesic $\gamma_k \colon [0, \ell_k] \to ( \R \times \Sph^{n-1}, g_k )$ starting at $( 0, \theta_k )$ and of length $\ell_k = 2\inj_{g_{\beta}}( x_k ) < 2\pi/\sqrt{ 1 + \varepsilon }$. Note that, by definition of the injectivity radius, $\gamma_k|_{ [ 0, \ell_k/2 ] }$ and $\gamma_k|_{ [\ell_k/2,\ell_k] }$ are length-minimizing. Thus $\gamma_k$ is contained in $(-2\pi,2\pi) \times \Sph^{n-1}$.

By passing to a further subsequence and relabeling, we may suppose that $\theta_k \rightarrow \theta \in \mathbb{S}^{n-1}$ and $v_k = \gamma_{k}'(0) \rightarrow v \in T( \mathbb{R} \times \Sph^{n-1} )$. By the stability of the geodesic equation, see e.g. \cite[Lemma~1.5]{Sakai-1983-on-continuity-of-injectivity-radius-function}, the geodesic in $\R \times \mathbb{S}^{n-1}$ starting at $(0,\theta)$ in the direction $v$ is a pointwise limit of $( \gamma_k )_k$ (see \cite[page~96]{Sakai-1983-on-continuity-of-injectivity-radius-function} for a similar argument). Obviously the limiting geodesic is also closed. As the injectivity radius of $\R \times \mathbb{S}^{n-1}$ is $\pi$ and the limiting closed geodesic has length $2r_0 < 2\pi$, we obtain a contradiction. The claim follows.

We next consider the case $p_\infty = \infty$ and $x_k \rightarrow p_\infty$. Now, whenever $( x_k )_k$ is a sequence in $\Omega = \R^n \setminus P$ converging to $p_\infty$ (i.e., eventually leaving each compact subset of $\R^n$), we claim that $\liminf_{k} \inj_{g_{\beta}}(x_k) \geq m^{1/\beta}\pi$. The argument is similar to the case above, so we only outline the necessary modifications. To this end, let $M = \max_{ j \neq \infty } |p_j|$ and consider
$$
	\Phi_{\infty,s}(u,\theta)=e^{s+u}\theta
$$
in $U_s = \{ (u,\theta) \in \R \times \mathbb{S}^{n-1} \colon u > -s + \log(3M) \}$.

As above,
\begin{equation*}
	\Phi_{\infty,s}^*g_\beta
	=
	A_{\infty,s}(u,\theta)^2
	\bigl(du^2+g_{\Sph^{n-1}}\bigr),
\end{equation*}
where
\begin{equation}
	A_{\infty,s}^{\beta}
	=
	\sum_{i=1}^m
	|\theta-e^{-(s+u)}p_i|^{-\beta}.
	\label{eq:infinity-factor}
\end{equation}
Indeed,
$A_{\infty,s}=e^{s+u}\rho_\beta\circ\Phi_{\infty,s}$. Observe that on every compact $u$-interval, each summand in~\eqref{eq:infinity-factor} converges to $1$ in $C^\infty_{\mathrm{loc}}(\R\times\Sph^{n-1})$ as $s \rightarrow \infty$.  Consequently,
\begin{equation*}
	A_{\infty,s}\longrightarrow m^{1/\beta}
	\quad\text{in }C^\infty_{\mathrm{loc}}
		(\R\times\Sph^{n-1})
	\quad\text{as }s\to\infty.
\end{equation*}
We pass to a subsequence for which $|x_k| > 3M$ for every $k$. Thus, there are parameters $( s_k )_k$ and $( \theta_k )_k$ in $\Sph^{n-1}$ such that $\Phi_{\infty,s_k}(0,\theta_k) = x_k$ and $s_k \rightarrow \infty$. Now the rest of the argument is similar to the case above. We omit the details.
\end{proof}
We next verify the conditions in Lemma~\ref{lem:g-length-of-conformal-closed-geodesic} for $F_{\beta}$. Since $b_\beta=\nabla F_\beta=\sum_{i=1}^m w_{i,\beta}b_i$,
the triangle inequality and~\eqref{eq:fi-derivatives} give
$$
|\nabla F_\beta|
=
|b_\beta|
\leq
\sum_{i=1}^m w_{i,\beta}|b_i|
=
\sum_{i=1}^m \frac{w_{i,\beta}}{r_i}.
$$
Since $\beta>0$, for every $i$ we have
$$
\rho_\beta
=
\left(
\sum_{j=1}^m r_j^{-\beta}
\right)^{1/\beta}
\geq
\left(r_i^{-\beta}\right)^{1/\beta}
=
\frac1{r_i}.
$$
Consequently,
$$
|\nabla F_\beta|
\leq
\rho_\beta
\sum_{i=1}^m w_{i,\beta}
=
\rho_\beta.
$$
\begin{proof}[Proof of Proposition~\ref{proposition-injectivity-radius-lower-bound}]
    Let $\pi - D_2 > \varepsilon > 0$ and first let $n \geq 3$. Consider the set $K_\beta$ from Lemma~\ref{lem:ends} and the compact set $C = K \cup K_\beta$. Then $\inj_{g_{\beta}}(p) > \pi-\varepsilon$ for every $p \in \Omega \setminus C$ by the lemma. Thus, if it were the case that $\inj_{g_{\beta}}(p) < D_2$ for some $p \in \Omega$, then necessarily $p \in C$. By the continuity of the pointwise injectivity radius, see \cite[Proposition~10.37]{Lee}, we therefore find $p \in C$ such that $\iota \coloneqq \inj_{g_{\beta}}(p) = \inf_{x \in \Omega} \inj_{g_\beta}(x)$.
    
    By Theorem~\ref{theorem-klingenberg-lemma} (and Rauch comparison which is applicable by Proposition~\ref{proposition-sectional-curvature-upper-bound}), we find a periodic geodesic $\gamma \colon \R \to ( \Omega, g_\beta )$ starting at $p$ and of length $2\iota$; the smoothness at $0$ follows because $\iota = \inj_{g_{\beta}}( \gamma( \iota ) )$. Now Lemma~\ref{lem:g-length-of-conformal-closed-geodesic} gives
	$$
		2\pi\leq2\iota
		<2D_2=\sqrt2\,\pi<2\pi.
	$$
	This is a contradiction, thus leading to the proof of~\eqref{eq:global-inj} in the case $n \geq 3$. Regarding the case $n = 2$, if we had $\iota < \pi - \varepsilon$, we argue as above and find $p \in C$ with $\iota = \iota_{g_{\beta}}( p )$ and a smooth geodesic $\gamma$ at $p$ of length $<2(\pi-\varepsilon)$. This leads to a contradiction as in the case $n \geq 3$. Since $\varepsilon$ can be taken to be arbitrarily small, the claim follows.
\end{proof}

\begin{proof}[Proof of Theorem~\ref{thm:finite-punctures-CAT-2}]
    The quasihyperbolic metric can be approximated, in the pointed Gromov--Hausdorff sense, by smooth Riemannian manifolds $( \Omega, k_\beta )$ for $\beta \rightarrow \infty$, see Proposition~\ref{proposition-smooth-approximation}. By Theorem~\ref{theorem-smooth-approximation}, the Riemannian approximations are $\mathrm{CAT}(2)$ for $n \geq 3$ and $\CAT(1)$ (and $\CAT(0)$ for balls with radius $\leq \pi/2$) for $n = 2$. Thus, the conclusion follows from the stability, Theorem~\ref{thm:CAT-stability}.
\end{proof}

\section{Proofs of Theorems~\ref{thm:universal-cat2}, \ref{theorem-cbb-equivalent-to-concavity}, and \ref{thm:convex-cat0}}\label{theorem-proof-of-many-results}

\subsection{The main result}
In this subsection, we prove Theorem~\ref{thm:universal-cat2}, i.e., that every quasihyperbolic domain is $\mathrm{CAT}(2)$ when $n \geq 3$. Lemma~\ref{lemma-monotonicity} and the stability theorem, Theorem~\ref{thm:CAT-stability}, reduces the claim to the case where $P = \R^n \setminus \Omega$ is finite. For such an $\Omega$, Theorem~\ref{thm:finite-punctures-CAT-2} shows the claim. The claim follows.

\begin{remark}\label{remark-small-balls-are-cat(0)-in-two-dimensions}
    The proof presented above shows more when $n = 2$. Indeed, then $(\Omega,k_\Omega)$ is $\CAT(1)$ and $B_{k_\Omega}(p,r)$ is $\CAT(0)$ for every $0 < r \leq \pi/2$. Consequently, for the radii $\pi/2 < r < \pi$, the ball is $\CAT(0)$ in the intrinsic length distance.
\end{remark}

\begin{remark}\label{remark-local-dc-regularity}
    Asplund proved in~\cite[page 235]{Asplund-1969-cebysev-sets-in-hilbert-space} that if $\Omega \subsetneq \R^n$ is a domain, then $x \mapsto \varphi(x) = \delta_{\Omega}(x)^2 - |x|^2$ is concave. In particular, it holds that $\delta_{\Omega}^{2} = \varphi - (-|\cdot|^2)$ represents $\delta_{\Omega}^2$ as a difference of two concave functions. That is, $\delta_\Omega^2$ is a \emph{$\mathcal{DC}$ function}; see \cite{Ambrosio-Bertrand-2018-dc-calculus}. Since every $\mathcal{C}^2$-regular function is locally $\mathcal{DC}$ and composition of locally $\mathcal{DC}$ functions remains in the same class, it holds that $\delta_{\Omega}$ and $\delta_{\Omega}^{-2}$ are locally $\mathcal{DC}$.
    
    As is well-known, see e.g. \cite{Ambrosio-Bertrand-2018-dc-calculus}, the distributional derivative of $\delta_{\Omega}$ (resp. $\delta_{\Omega}^{\pm 2})$ has bounded variation and thus its distributional Hessian is represented by a matrix-valued Radon measure. This allows us to prove the following \emph{distributional sectional curvature identity} for $g = \delta_{\Omega}^{-2}g_{\euc}$ for the Euclidean distance: for each $2$-plane $\Pi$, the distributional sectional curvature satisfies
    \begin{align}
        K_g(\cdot,\Pi)
        &=
        \delta_\Omega
        \operatorname{tr}_{\Pi}(\Hess\delta_\Omega)-\mathcal{L}^n\llcorner{\Omega},
        \label{eq:distributional-distance-curvature}
    \end{align}
    where $\mathcal{L}^n\llcorner{\Omega}$ is the Lebesgue measure on $\Omega$. Here \eqref{eq:distributional-distance-curvature} should be understood as equality between two locally finite Radon measures on $\Omega$. This follows from a standard smooth approximation via a partition of unity and convolution; see Lemma~\ref{lem:smooth-distance-curvature} below.

    A direct link between distributional sectional and Alexandrov curvature bounds for $\mathcal{C}^1$ metrics were recently established in \cite{eros-kunzinger-obanyan-vardabasso-2026-distributional-sectional-curvature-bounds-for-riemannian-metrics-of-low-regularity}. However, since $\delta_{\Omega}^{-2}$ is only (local) $\mathcal{DC}$, the applicability of their results seems barely out of reach for us. In fact, an Alexandrov curvature lower bound on a quasihyperbolic domain turns out to be equivalent to $\delta_{\Omega}$ being $\mathcal{C}^1$. We discuss this in the coming subsection.
\end{remark}

We finish this subsection with the following lemma.
\begin{lemma}\label{lem:smooth-distance-curvature}
Let $\Omega\subsetneq\R^n$ be a domain and $n \geq 2$. Then \eqref{eq:distributional-distance-curvature} holds in the sense of Radon measures.
\end{lemma}

\begin{proof}
For the proof, let $f = \delta_{\Omega}$. Since $f$ is locally $\mathcal{DC}$, its distributional derivative is a BV map (see e.g. \cite{Ambrosio-Bertrand-2018-dc-calculus}). Then standard approximation using a Whitney decomposition $\{ Q_i \}_{ i \in N }$ of $\Omega$, a partition of unity, and convolution argument leads to the following sequence of smooth functions $( f_j )_{ j \geq 1 }$: $f_j > 0$ in $\Omega$ and for $F = -\log(f)$ and $F_j = -\log( f_j )$, the following holds.
\begin{enumerate}
    \item $F_j \rightarrow F$ uniformly and $Df_j \rightarrow Df$ in $L^{2}$ in every compact set $K \subset \Omega$;
    \item $\Hess f_j \mathcal{L}^{n} \rightharpoonup \Hess f$ in the vague convergence of (matrix-valued) Radon measures on $\Omega$.
\end{enumerate}
We are ready to proceed with the proof. We apply Lemma~\ref{lem:conformal-curvature} with $F_j$.  By the chain rule and the convention for the tensor product,
\begin{equation*}
\Hess F_j
=
-\frac{\Hess f_j}{f_j}
+\frac{\nabla f_j \otimes\nabla f_j}
{f_j^2},
\quad
\nabla F_j
=
-\frac{\nabla f_j }{f_j}.
\end{equation*}

Let $e_1,e_2$ be a Euclidean orthonormal basis of $\Pi$. By
definition of the tensor product,
\begin{align*}
\operatorname{tr}_{\Pi}
\bigl(\nabla f_j \otimes \nabla f_j \bigr)_x
&=
\sum_{i=1}^2
\langle\nabla f_j(x),e_i\rangle^2
=
\left|\pi_{\Pi}\bigl(\nabla f_j(x)\bigr)\right|^2.
\end{align*}
Since $e^{2F_j}=f_j^{-2}$, the conformal-curvature formula
\eqref{eq:general-conformal-curvature} therefore gives
\begin{align*}
    \frac{ 1 }{ f_j^2(x) }K_{g_j}(x,\Pi)
    &=
    \frac{1}{f_j(x)}
    \operatorname{tr}_{\Pi}(\Hess f_j)_x
    -
    \frac{1}{f_j^2(x)} |\nabla f_j(x)|^2.
\end{align*}
Multiplying by
$f_j(x)^2$ proves
\begin{align*}
    K_{g_j}(x,\Pi)
    =
    f_{j}(x)\operatorname{tr}_{\Pi}(\Hess f_j)_x
    -
    \left|\nabla f_j(x)\right|^2.
\end{align*}
Letting $j \rightarrow \infty$, these converge vaguely to
\begin{align*}
    K_{g}( \cdot, \Pi )
    =
    f \operatorname{tr}_{\Pi}(\Hess f)
    -
    \mathcal{L}^n\llcorner{\Omega},
\end{align*}
where we used the fact that $|\nabla f| = 1$ almost everywhere in $\Omega$. This shows the claim.
\end{proof}

\subsection{Curvature lower bounds}
    In this section, we prove Theorem~\ref{theorem-cbb-equivalent-to-concavity}.

    We first prove '$(1) \implies (4)$': Suppose that $\Omega \subsetneq \R^n$ is concave, i.e., $C = \R^n \setminus \Omega$ is convex. Then, the nearest point projection $\pi_{C} \colon \R^n \to C$ is $1$-Lipschitz. Furthermore, since
    \begin{align*}
        \delta_{\Omega}(x) = |x-\pi_{C}(x)|
        \quad\text{for every $x \in \R^n$,}
    \end{align*}
    it holds that $\delta_{\Omega}$ is convex. At a differentiability point $x \in \Omega$ of $\delta_{\Omega}$, a direct computation shows
    \begin{align}\label{equation-projection-term}
        \nabla \delta_{\Omega}(x) = \frac{ x-\pi_{C}(x) }{ |x-\pi_{C}(x)| }.
    \end{align}
    This implies that $\delta_\Omega \in \mathcal{C}^{1,1}_{\text{loc}}(\Omega)$ because the right-hand side is locally Lipschitz (see e.g.~\cite[Corollary, page~251]{clarke-1975-generalized-gradients-and-applications}). The convexity of $\delta_{\Omega}$ implies that its distributional Hessian is positive semi-definite, so \eqref{eq:distributional-distance-curvature} immediately shows that the sectional curvature on each plane is bounded from below by $-1$. We derive a slightly stronger identity for the \emph{unnormalized curvature} in order to apply the approximation result in \cite[Theorem~3.1]{eros-kunzinger-obanyan-vardabasso-2026-distributional-sectional-curvature-bounds-for-riemannian-metrics-of-low-regularity}, thereby obtaining a local Alexandrov curvature lower bound of $-1$. We use the approximation from Lemma~\ref{lem:smooth-distance-curvature} to fill in the gap. Since $g=\delta_{\Omega}^{-2}g_{\mathrm{euc}}$, the induced metric on bivectors satisfies
$$
    g(X \wedge Y, X \wedge Y)=\delta_{\Omega}^{-4}(|X|^2|Y|^2-\langle X,Y\rangle^2).
$$
Using the notation from \cite{eros-kunzinger-obanyan-vardabasso-2026-distributional-sectional-curvature-bounds-for-riemannian-metrics-of-low-regularity}, we establish the following lemma for the unnormalized Riemannian curvature tensor.
\begin{lemma}\label{lemma-unnormalized-curvature-identity}
For every smooth vector fields $X$ and $Y$ on $\Omega$, the bivector $V = X \wedge Y$ satisfies the identity
\begin{gather}\label{eq:unnormalized-curvature-identity}
    \mathcal R_g(V,V)
+g(V,V)\,\mathcal L^n\\ \nonumber
\quad =
\delta_{\Omega}^{-3}\Big(
|Y|^2\operatorname{Hess}\delta_{\Omega}(X,X)
+|X|^2\operatorname{Hess}\delta_{\Omega}(Y,Y)
-2\langle X,Y\rangle\operatorname{Hess}\delta_{\Omega}(X,Y)
\Big)
\end{gather}
as Radon measures.
\end{lemma}
\begin{proof}
    We use the notation from the proof of Lemma~\ref{lem:smooth-distance-curvature}. In particular, $f = \delta_{\Omega}$. We apply the smooth conformal-curvature formula to the approximating metrics $g_j=f_j^{-2}g_{\mathrm{euc}}$. The formula for $\mathcal R_{g_j}(V,V)$ contains Hessian and gradient contributions. The gradient term is
$$
-f_j^{-4}
\bigl(|X|^2|Y|^2-\langle X,Y\rangle^2\bigr)
|\nabla f_j|^2\,\mathcal L^n.
$$
The approximation in Lemma~\ref{lem:smooth-distance-curvature} gives $f_j\to f>0$ locally uniformly and $\nabla f_j\to\nabla f$ strongly in $L^2_{\mathrm{loc}}$, hence  $|\nabla f_j|^2/f_j^4\to|\nabla f|^2/f^4=1/f^4$ in $L^1_{\mathrm{loc}}$. Thus the gradient contribution converges to $-g(V,V)\mathcal L^n$. On the other hand, the weighted Hessian contribution vaguely converge as Radon measures. These limits agree with the corresponding contraction of the distributional curvature tensor of $g$. We conclude the validity of \eqref{eq:unnormalized-curvature-identity}, noting that the smooth coefficients involving $X,Y$ cause no additional difficulty in these measure limits.
\end{proof}
Using \eqref{eq:unnormalized-curvature-identity}, we deduce
$$
\mathcal R_g(X\wedge Y,X\wedge Y)
\ge
-g(X\wedge Y,X\wedge Y)\,\mathcal L^n
$$
using the algebraic identity
$$
\begin{aligned}
&|Y|^2\operatorname{Hess}\delta_{\Omega}(X,X)
+|X|^2\operatorname{Hess}\delta_{\Omega}(Y,Y)
-2\langle X,Y\rangle\operatorname{Hess}\delta_{\Omega}(X,Y)\\
&\qquad =
\sum_{i=1}^n
\operatorname{Hess}\delta_{\Omega}
\bigl(Y^iX-X^iY,\;Y^iX-X^iY\bigr)
\ge0,
\end{aligned}
$$
where $X^i$ and $Y^i$ denote the Euclidean components of $X$ and $Y$, respectively.
A local Alexandrov curvature lower bound of $-1$ follows from \cite[Theorem~3.1]{eros-kunzinger-obanyan-vardabasso-2026-distributional-sectional-curvature-bounds-for-riemannian-metrics-of-low-regularity}. Hence $(\Omega, k_\Omega)$ is $\mathrm{CBB}(-1)$ by the Globalization Theorem \cite{AlexanderKapovitchPetrunin}.

    Clearly '$(4) \implies (3)$' holds. Next, we prove '$(3)\implies (2)$'. Suppose that $( \Omega, k_\Omega )$ satisfies a local variable Alexandrov curvature lower bound. By localizing to a small enough ball, $U = B_{k_\Omega}(p,r) \subset \Omega$, we have both an Alexandrov curvature upper and lower bounds on $( U, k_\Omega )$. For such bounded curvature spaces, Berestovski\u{\i} constructed in \cite{Berestovskij-1976-introduction-of-a-riemann-structure-into-certain-metric-spaces} distance coordinates and a continuous metric tensor $g$ with the following properties: the transition maps are locally bi-Lipschitz and in these coordinates, the metric $g$ has locally Lipschitz coefficients and the distance $k_\Omega$ is locally isometric to the Riemannian distance induced by $g$. Subsequently, Nikolaev constructed in \cite{nikolaev-1983-smoothness-of-the-metric-of-spaces-with-bilaterally-bounded-curvature-in-the-sense-of-alexandrov} harmonic coordinates whose transition maps are (at least) $\mathcal{C}^{2,\alpha}_{\text{loc}}$ and $g$ is $\mathcal{C}^{1,\alpha}_{\text{loc}}$ in the harmonic coordinates. To prove (2), we establish the following. 
    \begin{lemma}\label{lemma-regularity-of-the-metric}
    The metric $g$ is $\mathcal{C}^{1}$-regular with respect to the identity chart and $g = \delta_{\Omega}^{-2}g_{\euc}$.
    \end{lemma}
    \begin{proof}
    Fix a harmonic chart $\varphi \colon \R^n \supset W' \to ( W, g ) \subset ( U, g )$ for a domain $W$, and note by construction that the Riemannian metric $d_{ \varphi^{*}g }$ defined by $\varphi^{*}g$ makes $\varphi \colon ( W', d_{ \varphi^{*}g } ) \to ( W, k_{\Omega} )$ a length-preserving local isometry. If $\iota \colon ( W, k_\Omega ) \to W$ denotes the identity map, it follows that $F = \iota \circ \varphi \colon ( W', d_{\varphi^{*}g} ) \to W$ is locally bi-Lipschitz. By construction, we conclude that at each differentiability point of $F$, the differential of $F$ is angle-preserving and thus the distributional theory of angle-preserving maps by Iwaniec \cite{iwaniec-1982-regularity-theorems-for-solutions-of-partial-differential-equations-for-quasiconformal-mappings-in-several-dimensions} (see also \cite{liimatainen-salo-2014-n-harmonic-coordinates-and-the-regularity-of-conformal-mappings,julin-liimatainen-salo-2017-p-harmonic-coordinates-for-holder-metrics-and-applications}) implies $\mathcal{C}^{2,\alpha}_{\text{loc}}$-regularity of $F$. Consequently, $g = \delta_{\Omega}^{-2}g_{\euc}$ is $\mathcal{C}^{1,\alpha}_{\text{loc}}$-regular in $W$, implying the differentiability of~$\delta_{\Omega}$ in $W$ and on $\Omega$ by a covering argument.
    \end{proof}
    Lastly, we establish '$(2) \implies (1)$': We use the characterization of convexity in terms of the nearest point projection being single-valued. To conclude this, in case $x \in \Omega$ is a differentiability point of $\delta_{\Omega}$ and $z \in \partial \Omega$ satisfies $|x-z| = \delta_{\Omega}(x,z)$, as $|y-z| = \delta_{\Omega}(y)$ along the Euclidean segment from $z$ to $x$, we see that $\delta_{\Omega}$ satisfies \eqref{equation-projection-term} with $\pi_{\R^n \setminus \Omega}(x)$ replaced by $z$. This characterization implies that $z$ is unique at differentiability points of $\delta_\Omega$. Hence the metric projection $\pi_{\R^n \setminus \Omega} \colon \Omega \to \partial \Omega$ is well-defined at each point and the concavity of $\Omega$ follows.
\begin{remark}\label{remark-non-branching-geodesics-is-not-equivalent-to-differentiability}
    While an Alexandrov curvature lower bound for a quasihyperbolic domain $( \Omega, k_\Omega )$ implies that quasihyperbolic geodesics cannot branch, the converse is not true. Indeed, if $\Omega$ is the Euclidean unit ball, the quasihyperbolic geodesics in $(\Omega,k_\Omega)$ do not branch even though $(\Omega,k_\Omega)$ does not have an Alexandrov curvature lower bound. To rule out branching, it suffices to notice that the quasihyperbolic metric is smooth outside the origin and any quasihyperbolic geodesic passing through the origin is radial. Since $\Omega$ is not concave, an Alexandrov curvature lower bound does not hold by Theorem~\ref{theorem-cbb-equivalent-to-concavity}.
\end{remark}

\subsection{Quasihyperbolic metrics on convex domains} To prove Theorem~\ref{thm:convex-cat0}, we first establish the following result.
\begin{proposition}\label{prop:convex-equivalence}
Let $\Omega\subsetneq\R^n$ be a domain. Then the following are
equivalent
\begin{enumerate}
    \item $\Omega$ is convex;
    \item $-\log\delta_\Omega$ is convex on $\Omega$.
\end{enumerate}
\end{proposition}
\begin{proof}
We prove '$(1)\implies (2)\implies (1)$'.

We prove '$(1)\implies (2)$' first using a standard argument. Let $x,y\in\Omega$ and $t\in(0,1)$. For
$0<r<\delta_\Omega(x)$ and $0<s<\delta_\Omega(y)$, convexity of
$\Omega$ gives
\begin{equation*}
 (1-t)B(x,r)+tB(y,s)
 =B\bigl((1-t)x+ty,(1-t)r+ts\bigr)\subset\Omega,
\end{equation*}
where we use standard Minkowski summation $A+B = \{ a+b \colon (a,b) \in A \times B \}$. Letting $r\uparrow\delta_\Omega(x)$ and
$s\uparrow\delta_\Omega(y)$ yields
\begin{equation*}
 \delta_\Omega((1-t)x+ty)
 \geq(1-t)\delta_\Omega(x)+t\delta_\Omega(y).
\end{equation*}
Thus $\delta_\Omega$ is concave. Since $-\log$ is decreasing and
convex,
\begin{align*}
 -\log\delta_\Omega((1-t)x+ty)
 &\leq-\log\bigl((1-t)\delta_\Omega(x)
                  +t\delta_\Omega(y)\bigr)\\
 &\leq(1-t)(-\log\delta_\Omega(x))
       +t(-\log\delta_\Omega(y)).
\end{align*}
Hence $-\log\delta_\Omega$ is convex on $\Omega$.

We finish the proof with '$(2)\implies (1)$'. Set $u=-\log\delta_\Omega$ and, for $r>0$, define 
\begin{equation*}
\Omega_r
=
\{x\in\R^n \colon\delta_\Omega(x)\geq r\}
=
\{x\in\Omega\colon u(x)\leq-\log r\}.
\end{equation*}
Consider $x,y\in\Omega$. Since $\Omega_{r}$ is convex as a sublevel set of a convex function, the Euclidean segment $[x,y]$ is contained in $\Omega_r \subset \Omega$ for small enough $r$. Hence $\Omega$ is convex. The proof is complete.    
\end{proof}

\begin{remark}
    By a slight modification of the proof of Proposition~\ref{prop:convex-equivalence}, the local convexity of $u = -\log \delta_{\Omega}$ improves to convexity of $\Omega$. Indeed, each pair $x, y \in \Omega$ is contained in the same path-connected component $C_r$ of some closed sublevel set $\Omega_r = \{ x \in \Omega \colon u(x) \leq -\log r \}$. The local convexity of $u$ implies that $C_r$ is locally convex which, by the Tietze–Nakajima theorem \cite{tietze-1928,Nakajima-1928}, implies that $C_r$ is convex. In particular, the convexity of $\Omega$ follows by the arbitrariness of $r$.
\end{remark}

\begin{proof}[Proof of Theorem \ref{thm:convex-cat0}]
Fix $p\in\Omega$, and, for $j\in\mathbb{N}$, define $\Omega_j=\Omega\cap B(p,j)$. Then $(\Omega_j)_j$ is an increasing sequence of bounded convex domains containing $p$, and $\bigcup_{j=1}^{\infty}\Omega_j=\Omega$. Moreover, for $x\in\Omega_j$,
\begin{equation}
\delta_{\Omega_j}(x)
=
\min\{\delta_\Omega(x),j-|x-p|\}.
\label{eq:truncated-boundary-distance}
\end{equation}

Since $\Omega_j$ is convex, the function $f_j=-\log\delta_j$ is continuous and convex by Proposition~\ref{prop:convex-equivalence}, and bounded below. Since $e^{f_j}=1/\delta_{\Omega_j}$, the conformally deformed metric by the density $e^{f_j}$ is precisely
$k_{\Omega_j}$. Therefore, \cite[Theorem~1]{LytchakStadler} gives that
$(\Omega_j, k_{\Omega_j})$ is $\CAT(0)$. We next prove that
$$
(\Omega_j,k_j,p)
\xrightarrow{\mathrm{pGH}}
(\Omega,k_\Omega,p).
$$
It suffices to show that, for each $R>0$, the identity map $K_{R} \coloneqq \overline{B}_{k_\Omega}(p,R) \to \overline{B}_{k_{\Omega_j}}(p,R)$ is a surjective isometry for large enough $j$. To this end, the quasihyperbolic domain is proper, so $K_{3R}$ is a compact subset
of $\Omega$. Hence, for all sufficiently large $j$,
$$
j>
\max_{x\in K_{3R}}
\bigl(|x-p|+\delta_\Omega(x)\bigr).
$$
By \eqref{eq:truncated-boundary-distance}, it follows that $\delta_{\Omega_j}=\delta_{\Omega}$ on $K_{3R}$. Since $\Omega_j\subset\Omega$ and $\delta_{\Omega_j}\leq\delta_\Omega$, we have
\begin{equation}
k_\Omega(x,y)\leq k_{\Omega_j}(x,y),
\quad x,y\in\Omega_j.
\label{eq:qhm-domain-monotonicity}
\end{equation}
By triangle inequality, if $x, y \in K_R$, the $k_\Omega$-geodesic $\gamma$ joining $x$ to $y$ lies in $K_{3R}$.
Thus, \eqref{eq:qhm-domain-monotonicity} holds with an equality. This implies the claim.
\end{proof}

\section{Consequences of Theorem~\ref{thm:universal-cat2}}\label{section-consequences-of-theorems}
\subsection{Conjectures of Väisälä}
For the remainder of the subsection, we fix a domain $\Omega \subsetneq\R^n$ and recall that $( \Omega, k_\Omega )$ is $\mathrm{CAT}(2)$.
\begin{theorem}[Uniqueness of geodesics]\label{theorem-uniqueness-of-geodesics}
Geodesics in $(\Omega, k_\Omega)$ are unique up to length $D_2 = \pi/\sqrt{2}$. That is, if $k_\Omega(x,y) < D_2$, then there exists a unique quasihyperbolic geodesic joining $x$ to $y$ in $\Omega$.
\end{theorem}
\begin{proof}
    The claim holds in all $\mathrm{CAT}(2)$ spaces, see e.g. \cite[Chapter II, Proposition 1.4]{BridsonHaefliger}.
\end{proof}

\begin{theorem}[Extendability of geodesics]\label{theorem-extendability-of-geodesics}
Geodesics in $(\Omega,k_\Omega)$ are extendable: for every local geodesic $\gamma \colon [a,b] \to ( \Omega, k_\Omega )$ parametrized by arc length, there exists a local isometry $\bar{\gamma} \colon \R \to ( \Omega, k_\Omega )$ extending $\gamma$. Moreover, $\bar{\gamma}$ is length-minimizing on each segment of length $< D_2$.
\end{theorem}
\begin{proof}
    Since $(\Omega, k_\Omega)$ is homeomorphic to a topological manifold without boundary and $\CAT(2)$, it has the geodesic extension property, see \cite[Chapter II, Proposition 5.12 and Lemma 5.8 (1)]{BridsonHaefliger}. The segments of $\bar{\gamma}$ up to length $< D_2$ are length-minimizing by \cite[Chapter II, Proposition 1.4 (2)]{BridsonHaefliger}.
\end{proof}

\begin{theorem}[Convexity of balls]\label{theorem-convexity-of-balls}
For every $p \in \Omega$, the quasihyperbolic ball $B_{k_\Omega}( p, D_2/2 )$ is quasihyperbolically convex.
\end{theorem}
\begin{proof}
    The claim is \cite[Chapter II, Proposition (3)]{BridsonHaefliger}.
\end{proof}

\begin{remark}\label{remark-vaisalas-conjectures}
    The proofs clearly show that if a quasihyperbolic domain $( \Omega, k_\Omega )$ is $\mathrm{CAT}(\kappa)$ for $\kappa \in \R$, then we may replace $D_2$ by $D_\kappa$ in the theorems above. In particular, this gives a sharper result on two-dimensional domains (which are $\mathrm{CAT}(1)$) or on convex domains (which are $\CAT(0)$).
\end{remark}

In the following subsections, we derive a number of consequences from Theorems~\ref{theorem-uniqueness-of-geodesics}-\ref{theorem-convexity-of-balls}. 
\subsection{First order regularity of quasihyperbolic geodesics}\label{section-first-order-regularity-of-qh-geodesics}

In this subsection, we recall and refine estimates for quasihyperbolic geodesics due to \cite{Martin1985-qc-and-bi-lipschitz-homeomorphisms-uniform-domains-and-the-quasihyperbolic-metric}. While Martin proves sharp regularity for individual geodesics, our contribution concerns the regularity when varying an end point of the geodesics. Our methods use in a critical way the $\CAT(\kappa)$ theory.

We first fix some notation and definitions. Let $(\Omega,k_\Omega)$ be a quasihyperbolic domain that is
$\CAT(\kappa)$ and let $U_p=B_{k_{\Omega}}(p, D_\kappa)\setminus \{p\}$.

\begin{enumerate}
    \item For every $x\in U_p$, let $\gamma_x:[0,1]\to\Omega$
be the quasihyperbolic geodesic from $p$ to $x$, parametrized with
constant quasihyperbolic speed. Thus $\gamma_x(0)=p$, $\gamma_x(1)=x$, and
$$\frac{|\gamma_x'(t)|}
{\delta_\Omega(\gamma_x(t))}
=
k_\Omega(p,x)$$
for (almost) every $t\in[0,1]$.
\item Let $L_x=\ell(\gamma_x)$ denote the Euclidean length of $\gamma_x$, and let $C_x:[0, L_x]\to\Omega$
be its Euclidean arclength parametrization, oriented from $p$ to $x$. Thus
$$C_x(0)=p,
\quad
C_x(L_x)=x,
\quad
|C_x'(s)|=1.$$

\item Define $E_x:[0,1]\to\Omega$, $E_x(t)=C_x(L_xt)$. Then $E_x$ parametrizes the same oriented curve from $p$ to $x$ with
constant Euclidean speed:
$$E_x(0)=p,
\quad
E_x(1)=x,
\quad
|E_x'(t)|=L_x.$$
\item Define
$$\theta_x(t)
=
\frac1{L_x}
\int_0^t|\gamma_x'(s)|\,ds.$$
We conclude that $\theta_x:[0,1]\to[0,1]$ is a strictly increasing homeomorphism and $\gamma_x=E_x\circ\theta_x$. Equivalently, $E_x=\gamma_x\circ\theta_x^{-1}$. Moreover,
\begin{equation}\label{eq:theta_x'}
\theta_x'(t)
=
\frac{|\gamma_x'(t)|}{L_x}
=
\frac{k_\Omega(p,x)}{L_x}
\delta_\Omega(\gamma_x(t))
\end{equation}
for (almost) every $t$.
\end{enumerate}

We also consider uniform norms for paths and their derivatives. To fix notation, for $h \colon [0,1]\to\mathbb R^n$, define
\begin{align*}
    \|h\|_{\mathcal{C}^0([0,1];\mathbb R^n)} 
    &\coloneqq 
    \sup_{t\in[0,1]}|h(t)|,
    \quad\text{and}
    \\
    \|h\|_{\mathcal{C}^1([0,1];\mathbb R^n)} 
    &\coloneqq 
    \|h\|_{\mathcal{C}^0([0,1];\mathbb R^n)}
    +
    \|h'\|_{\mathcal{C}^0([0,1];\mathbb R^n)}.
\end{align*}
When the domain and codomain are clear, we write $\|\cdot\|_{C^0}$ and $\|\cdot\|_{C^1}$ to simplify notation.

To formulate precise statements, we also define the global Lipschitz constant and the pointwise upper Lipschitz constant for functions defined on an interval. For a map $h \colon I\to \mathbb{R}^n$, where $I\subset \R$ is an interval, its global Lipschitz constant is defined by
$$
\operatorname{Lip}(h)\coloneqq
\sup_{\substack{s,t\in I\\s\neq t}}
\frac{|h(s)-h(t)|}{|s-t|}.
$$
Its pointwise upper Lipschitz constant at $t\in I$ is
$$
\operatorname{lip}(h)(t)\coloneqq
\limsup_{\substack{s\to t\\s\in I,\,s\neq t}}
\frac{|h(s)-h(t)|}{|s-t|}.
$$
Since the domain of $h$ is an interval, a simple argument implies $\operatorname{Lip}(h) = \sup_{t} \operatorname{lip}(h)(t)$ in $[0,\infty]$.

We first establish some basic properties of quasihyperbolic geodesics.

\begin{proposition}\label{proposition-uniform-c11-geodesics}
Let $r>0$. Then, for every $x\in U_p$,
$$\gamma_x\in \mathcal{C}^{1,1}([0,1];\mathbb R^n) \text{ and } |\gamma_x'(t)|
=
k_{\Omega}(p,x)\delta_\Omega(\gamma_x(t))$$
for every $t\in[0,1]$, with one-sided derivatives at the endpoints. More precisely,
$$\|\gamma_x-p\|_{\mathcal{C}^0}
\leq
\delta_\Omega(p)(e^{k_{\Omega}(p,x)}-1),$$
$$\|\gamma_x'\|_{\mathcal{C}^0}
\leq
k_{\Omega}(p,x)e^{k_{\Omega}(p,x)}\delta_\Omega(p),$$
and
$$\operatorname{Lip}(\gamma_x')
\leq
2k_{\Omega}(p,x)^2e^{k_{\Omega}(p,x)}\delta_\Omega(p).$$
Consequently, if $r < D_\kappa$,
$$\sup_{x\in\overline B_{k_\Omega}(p,r)\setminus\{p\}}
\operatorname{Lip}(\gamma_x')
\leq
2r^2e^r\delta_\Omega(p).$$
\end{proposition}

To prove Proposition \ref{proposition-uniform-c11-geodesics}, we need the following lemma, essentially due to Martin \cite{Martin1985-qc-and-bi-lipschitz-homeomorphisms-uniform-domains-and-the-quasihyperbolic-metric}.

\begin{lemma}\label{lemma-lipschitz-constant-of-qh-geodesics}
Let $( \Omega, k_\Omega )$ be a quasihyperbolic domain and $B = \overline{B}_{k_\Omega}(p,r)$ for $r > 0$. If $C \colon [a,b] \to ( B, k_\Omega )$ is a quasihyperbolic geodesic parametrized by Euclidean unit speed, then the (one-sided) tangent vector $C' \colon [a,b] \to \Sph^{n-1}$ (at the end points) is Lipschitz with a pointwise Lipschitz constant
\begin{align*}
        \operatorname{lip}( C' )(t)
        \leq
        \frac{ 1 }{ \delta_{\Omega}(C(t)) }.
    \end{align*}
Consequently,
$$\operatorname{Lip}(C')
\leq
\frac{1}
{\operatorname{dist}
\bigl(B,\partial\Omega\bigr)}
\leq
\frac{e^r}{\delta_\Omega(p)}.$$
\end{lemma}

\begin{proof}
    Martin shows during the proof of \cite[Theorem 4.3]{Martin1985-qc-and-bi-lipschitz-homeomorphisms-uniform-domains-and-the-quasihyperbolic-metric} that
    \begin{align*}
        \frac{ | C'(t)-C'(s) | }{ |t-s| }
        &\leq
        \frac{ 4 (1+o(1)) }{ \delta_{\Omega}( C(t) ) }
        \frac{ | C(t) - C(s) | }{ |t-s| }
        \leq
        \frac{ 4 (1+o(1)) }{ \delta_{\Omega}( C(t) ) }
    \end{align*}
    for a function $o(1) \rightarrow 0$ as $s \rightarrow t$; here $\phi$ is the angle between the normal hyperplanes to $C$ at $C(s)$ and $C(t)$. This implies
    \begin{align*}
        \operatorname{lip}( C' )(t) \leq \frac{ 4 }{ \delta_{\Omega}(C(t)) }
    \end{align*}
    for every $t \in (a,b)$. Hence $C' \colon (a,b) \to \Omega$ has the stated Lipschitz constant up to a factor of four so, in particular, $C'$ extends to $[a,b]$ as a Lipschitz map. By the fundamental theorem of calculus, the extension of $C'$ coincides with the one-sided derivatives of $C$ at the end points. With a finer analysis, the factor of four can be removed, see \cite[Section~4.12, equation (4.10)]{Martin1985-qc-and-bi-lipschitz-homeomorphisms-uniform-domains-and-the-quasihyperbolic-metric} (or \cite[Theorem 4.15, Correction 4.16.]{Vaisala-2013}).
\end{proof}

\begin{proof}[Proof of Proposition \ref{proposition-uniform-c11-geodesics}]
Fix $x\in U_p$. By
Lemma~\ref{lemma-lipschitz-constant-of-qh-geodesics},
$C_x\in \mathcal{C}^{1,1}([0,L_x];\R^n)$. Since
$\gamma_x=C_x\circ(L_x\theta_x)$, by \eqref{eq:theta_x'}, by \eqref{eq:qh-estimate-1}, and 
$k_\Omega(p,\gamma_x(t))=t k_\Omega(p,x)$, we deduce
$$
e^{-t k_\Omega(p,x)}\delta_\Omega(p)
\leq
\delta_\Omega(\gamma_x(t))
\leq
e^{t k_\Omega(p,x)}\delta_\Omega(p).
$$
It follows that $\gamma_x$ is Lipschitz. Therefore
$\delta_\Omega\circ\gamma_x$ is Lipschitz, so
$\theta_x\in \mathcal{C}^{1,1}([0,1])$. Consequently,
$\gamma_x=C_x\circ(L_x\theta_x)$ belongs to
$\mathcal{C}^{1,1}([0,1];\R^n)$, and
$$
\gamma_x'(t)
=
k_\Omega(p,x)\delta_\Omega(\gamma_x(t))
C_x'(L_x\theta_x(t)).
$$
Since $|C_x'|=1$, this gives
$$
|\gamma_x'(t)|
=
k_\Omega(p,x)\delta_\Omega(\gamma_x(t))
$$
for every $t\in[0,1]$, with one-sided derivatives at the endpoints. The preceding boundary-distance estimate now gives
$$
\begin{aligned}
|\gamma_x(t)-p|
\leq
\int_0^t|\gamma_x'(u)|\,du 
&\leq
k_\Omega(p,x)\delta_\Omega(p)
\int_0^t e^{u k_\Omega(p,x)}\,du \\
&=
\delta_\Omega(p)
\bigl(e^{t k_\Omega(p,x)}-1\bigr).
\end{aligned}
$$
The desired estimates for $\gamma_x$ and $\gamma_{x}'$ are immediate.

To finish, for almost every $t\in[0,1]$,
$$
|C_x''(L_x\theta_x(t))|
\leq
\frac1{\delta_\Omega(\gamma_x(t))}
$$
and
$$
|\theta_x''(t)|
\leq
\frac{k_\Omega(p,x)}{L_x}
|\gamma_x'(t)|
=
\frac{k_\Omega(p,x)^2}{L_x}
\delta_\Omega(\gamma_x(t)).
$$
Since $\theta_x'$ is strictly positive, the second-order chain rule applies almost everywhere and gives
$$
\begin{aligned}
|\gamma_x''(t)|
&\leq
L_x^2
|C_x''(L_x\theta_x(t))|
\theta_x'(t)^2
+
L_x|\theta_x''(t)| \\
&\leq
2k_\Omega(p,x)^2
\delta_\Omega(\gamma_x(t))
\leq
2k_\Omega(p,x)^2e^{k_\Omega(p,x)}
\delta_\Omega(p).
\end{aligned}
$$
Thus,
$$
\operatorname{Lip}(\gamma_x')
\leq
2k_\Omega(p,x)^2e^{k_\Omega(p,x)}
\delta_\Omega(p).
$$
Finally, since $s\mapsto s^2e^s$ is increasing on
$[0,\infty)$, we obtain
$$
\sup_{x\in\overline B_{k_\Omega}(p,r)\setminus\{p\}}
\operatorname{Lip}(\gamma_x')
\leq
2r^2e^r\delta_\Omega(p).
$$
The proof is complete.
\end{proof}

Next we establish the first order continuity of quasihyperbolic geodesics.
\begin{proposition}
\label{proposition-first-order-endpoint-continuity}
For every compact set $K\Subset U_p$, there exists a constant $C_K<\infty$ such that
$$\|\gamma_x-\gamma_y\|_{\mathcal{C}^1}
\leq
C_K|x-y|^{1/2}$$
for every $x,y\in K$. Equivalently, the map $x\mapsto \gamma_x$ from $U_p$ into $\mathcal{C}^1([0,1];\mathbb R^n)$ is locally $\frac12$-Hölder continuous.
\end{proposition}

To prove Proposition \ref{proposition-first-order-endpoint-continuity}, we first establish some auxiliary lemmas. 

\begin{lemma}\label{lemma-cat-endpoint-estimate}
For $0 < r < D_\kappa$, there exists a constant $A_{\kappa,r}<\infty$ depending only on $\kappa$ and $r$ such that
$$\sup_{t\in[0,1]}
k_\Omega\bigl(\gamma_x(t),\gamma_y(t)\bigr)
\leq
A_{\kappa,r}k_\Omega(x,y)$$
for every $x,y\in\overline B_{k_\Omega}(p,r) \setminus \{p\}$.
In particular,
$$\|\gamma_x-\gamma_y\|_{\mathcal{C}^0}
\leq e^{4r}A_{\kappa,r}|x-y|$$
for every $x,y\in\overline B_{k_\Omega}(p,r) \setminus \{p\}$.
\end{lemma}

\begin{proof}
If $\kappa\leq0$, then $(\Omega,k_\Omega)$ is $\CAT(0)$. As such, the length distance is convex, i.e.,
$$
k_\Omega\bigl(\gamma_x(t),\gamma_y(t)\bigr)
\leq
t\,k_\Omega(x,y)
\leq
k_\Omega(x,y)
$$
by triangle comparison; see e.g. \cite[Proposition~II.2.2]{BridsonHaefliger}. This gives this special case.

Suppose that $\kappa>0$. The idea is again to use comparison triangles and the spherical law of cosines. To obtain explicit estimates, we apply \cite[Corollary~9.13]{AlexanderKapovitchPetrunin} on the ball
$B_{k_\Omega}\bigl(p,(D_\kappa+r)/2\bigr)$ and a subdivision
argument along the geodesic from $x$ to $y$. If also
$$
k_\Omega(x,y)<\frac{D_\kappa-r}{2},
$$
this leads to
$$
k_\Omega\bigl(\gamma_x(t),\gamma_y(t)\bigr)
\leq
\sec\left(\frac{\sqrt\kappa r}{2}\right)k_\Omega(x,y).
$$
If instead $k_\Omega(x,y)\geq(D_\kappa-r)/2$, we have
$$
k_\Omega\bigl(\gamma_x(t),\gamma_y(t)\bigr)
\leq
2r
\leq
\frac{4r}{D_\kappa-r}k_\Omega(x,y).
$$
Consequently, we may take
$$
A_{\kappa,r}
=
\max\left\{
\sec\left(\frac{\sqrt\kappa r}{2}\right),
\frac{4r}{D_\kappa-r}
\right\}.
$$

Finally, a standard calculation gives, for $x,y,u,v\in\overline B_{k_\Omega}(p,r)$,
\begin{equation}\label{eq:CAT-end-point-eq-1}
    k_\Omega(x,y)
\leq
\frac{2re^r}
{(1-e^{-2r})\delta_\Omega(p)}
|x-y|
\end{equation}
and
\begin{equation}\label{eq:CAT-end-point-eq-2}
    |u-v|
\leq
e^r\delta_\Omega(p)
\frac{e^{2r}-1}{2r}
k_\Omega(u,v).
\end{equation}
Together with the first part of the claim, we apply \eqref{eq:CAT-end-point-eq-1} and \eqref{eq:CAT-end-point-eq-2} with $u=\gamma_x(t)$ and $v=\gamma_y(t)$ to obtain
$$
|\gamma_x(t)-\gamma_y(t)|
\leq
e^{4r}A_{\kappa,r}|x-y|.
$$
Taking the supremum over $t\in[0,1]$ gives the desired result.
\end{proof}

\begin{remark}
When $\kappa>0$, the constants in
Lemma~\ref{lemma-cat-endpoint-estimate} need not remain
bounded as $r\uparrow D_\kappa$ (even in a sphere). Thus, uniform estimates hold only on compact
subsets of $U_p$.
\end{remark}

Next, we establish a Landau-type interpolation estimate.
\begin{lemma}[Landau-type interpolation]
\label{lemma-landau-interpolation}
If $h\in \mathcal{C}^{1,1}([0,1];\mathbb R^n)$, then
$$\|h'\|_{\mathcal{C}^0}
\leq
8\left(
\|h\|_{\mathcal{C}^0}
+
\|h\|_{\mathcal{C}^0}^{1/2}
\operatorname{Lip}(h')^{1/2}
\right).$$
\end{lemma}

\begin{proof}
For every $t\in[0,1]$, choose $\varepsilon\in\{-1,1\}$ so that $t+\varepsilon s\in[0,1]$ for $0< s\leq 1/2$. Then, by Taylor expansion,
$$
|h(t+\varepsilon s)-h(t)-\varepsilon s h'(t)|
\leq
\frac{\operatorname{Lip}(h')}{2}s^2,
$$
and consequently
$$
|h'(t)|
\leq
\frac{2\|h\|_{\mathcal{C}^0}}{s}+\frac{\operatorname{Lip}(h')}{2}s.
$$
If $\operatorname{Lip}(h')\leq 16 \|h\|_{\mathcal{C}^0}$, taking $s=1/2$ gives $|h'(t)|\leq 8\|h\|_{\mathcal{C}^0}$. If $\operatorname{Lip}(h')>16\|h\|_{\mathcal{C}^0}$, taking $s=2\sqrt{\|h\|_{\mathcal{C}^0}/\operatorname{Lip}(h')}$ gives $|h'(t)|\leq 2\sqrt{\|h\|_{\mathcal{C}^0}\operatorname{Lip}(h')}$. Hence
$$
\|h'\|_{\mathcal{C}^0}
\leq
8\left(
\|h\|_{\mathcal{C}^0}
+
\|h\|_{\mathcal{C}^0}^{1/2}\operatorname{Lip}(h')^{1/2}
\right).
$$   
The proof is complete.
\end{proof}

\begin{proof}[Proof of Proposition~\ref{proposition-first-order-endpoint-continuity}]
Choose $r<D_\kappa$ such that $K\subset\overline B_{k_\Omega}(p,r)$. For simplicity, write $B=e^{4r}A_{\kappa,r}$. By
Lemma~\ref{lemma-cat-endpoint-estimate},
$$
\|\gamma_x-\gamma_y\|_{\mathcal{C}^0}
\leq
B|x-y|.
$$
Moreover, Proposition~\ref{proposition-uniform-c11-geodesics} gives
$$
\operatorname{Lip}(\gamma_x')
+
\operatorname{Lip}(\gamma_y')
\leq
4r^2e^r\delta_\Omega(p).
$$
Applying Lemma~\ref{lemma-landau-interpolation} to
$\gamma_x-\gamma_y$, we obtain
$$
\begin{aligned}
\|\gamma_x'-\gamma_y'\|_{\mathcal{C}^0}
&\leq
8B|x-y| +
8B^{1/2}
\bigl(4r^2e^r\delta_\Omega(p)\bigr)^{1/2}
|x-y|^{1/2}.
\end{aligned}
$$
Since $K$ is bounded, the terms of order $|x-y|$ can be absorbed into
a constant multiple of $|x-y|^{1/2}$. Combining the preceding
estimates gives
$$
\|\gamma_x-\gamma_y\|_{\mathcal{C}^1}
\leq
C_K|x-y|^{1/2}.
$$
The claim follows.
\end{proof}

\subsection{Quasihyperbolic spheres}
In this subsection, we study the regularity of quasihyperbolic spheres using the results from Section~\ref{section-first-order-regularity-of-qh-geodesics}.

Recall that $(\Omega,k_\Omega)$ is a $\CAT(\kappa)$ quasihyperbolic domain and we consider $U_p=B_{k_{\Omega}}(p, D_\kappa)\setminus \{p\}$.

We fix further notation. For $x\in U_p$, define the terminal Euclidean unit tangent by
$$
\nu_x
=
\frac{\gamma_x'(1)}{|\gamma_x'(1)|}=\frac{E_x'(1)}{L_x}=C_x'(L_x).
$$
Define the quasihyperbolic distance function $f \colon U_p \to ( 0, D_\kappa )$ by $f(x)=k_{\Omega}(p,x)$. Then $f(\gamma_x(t))=tf(x)$ and $\gamma_x'(t)=f(x)\delta_{\Omega}(\gamma_x(t))$. Note that $|\gamma_x'(1)|=f(x)\delta_\Omega(x)>0$ so $x \mapsto \nu_x$ is well-defined in $U_p$. 

We next establish the regularity of the quasihyperbolic distance function $f$.

\begin{theorem}
\label{theorem-qh-distance-c11half}
The map $\nu$ is locally $\frac12$-Hölder continuous. Moreover,
$f\in \mathcal{C}_{\mathrm{loc}}^{1,\frac12}(U_p)$ and
$$
\nabla f(x)
=
\frac{\nu_x}{\delta_\Omega(x)},
\quad
|\nabla f(x)|
=
\frac1{\delta_\Omega(x)}
$$
for every $x\in U_p$. More precisely, for every $K\Subset U_p$, there exists $C_K<\infty$
such that
$$
|\nu_x-\nu_y|
+
|\nabla f(x)-\nabla f(y)|
\leq
C_K|x-y|^{1/2}
$$
for every $x,y\in K$.
\end{theorem}

\begin{proof}
Let $K\Subset U_p$. Since
$|\gamma_x'(1)|=f(x)\delta_\Omega(x)$ is bounded away from zero on
$K$, Proposition~\ref{proposition-first-order-endpoint-continuity}
and the local Lipschitz continuity of normalization $v/|v|$ on
$\mathbb R^n\setminus\{0\}$ give
$$
|\nu_x-\nu_y|
\leq
C_K|x-y|^{1/2}.
$$
The triangle inequality for $k_\Omega$ implies that the function $f$ is locally Lipschitz in the Euclidean sense. Consequently, Rademacher’s theorem implies that $f$ is differentiable almost everywhere. Note that at differentiability points of $f$, the gradient of $\nabla f(x)$ has length at most $1/\delta_{\Omega}(x)$ because
\begin{align*}
    \limsup_{ \delta \rightarrow 0^{+} }
    \frac{ |f(x+\delta v)-f(x)| }{\delta}
    \leq
    \limsup_{ \delta \rightarrow 0^{+} }
    \frac{ k_\Omega(x+\delta v,x) }{ \delta }
    =
    \frac{ |v| }{ \delta_{\Omega}(x) }.
\end{align*}
In fact, as $f(\gamma_x(t))=tf(x)$, we conclude that $\langle \nabla f(x), \gamma_x'(1) \rangle=f(x)$. Since $\gamma_x'(1)=f(x)\delta_\Omega(x)\nu_x$, it follows that
$$
\nabla f(x)
=
\frac{\nu_x}{\delta_\Omega(x)}
$$
at differentiability points of $f$. The right-hand side is locally $\frac12$-Hölder continuous because
$\nu$ is locally $\frac12$-Hölder and $\delta_\Omega^{-1}$ is locally Lipschitz. Using the continuity of $\nu_x/\delta_{\Omega}(x)$, standard arguments imply that $f \in \mathcal{C}^{1,\frac{1}{2}}_{\text{loc}}( U_p )$ (see e.g.~\cite[Corollary, page~251]{clarke-1975-generalized-gradients-and-applications}). 
\end{proof}

\begin{lemma}\label{lemma-properties-of-f}
Suppose that $( \Omega, k_\Omega )$ is $\mathrm{CAT}(\kappa)$ for $\kappa \leq 2$. For every $0 < r < D_\kappa$, there exists a $\mathcal{C}^{1,\frac12}_{\text{loc}}$-diffeomorphism $\Psi = ( \Psi_1, \Psi_2 ) \colon U_p \to (0,D_\kappa) \times f^{-1}(r)$ such that $f = \Psi_1$.
\end{lemma}

\begin{proof}
For $0<a\leq b<D_\kappa$, the standard estimates
$$
\delta_\Omega(x)\geq e^{-b}\delta_\Omega(p),
\quad
|x-p|\leq\delta_\Omega(p)(e^b-1)
$$
imply that $f^{-1}([a,b])$ is a compact
subset of $U_p$. This shows that $f$ is proper. If $w \in \partial \Omega$ satisfies $\delta_{\Omega}(p) = |w-p|$, then $t \mapsto f(p+tw)$ is unbounded on the interval $(0,1)$. This gives that $f$ is surjective.

By Theorem~\ref{theorem-qh-distance-c11half}, we have $|\nabla f(x)|=1/\delta_{\Omega}(x)>0$. Therefore $f$ is a $\mathcal{C}_{\text{loc}}^{1,\frac12}$-regular submersion. The properness of $f$ and the implicit function theorem implies that $f^{-1}(r)$ is a compact embedded $\mathcal{C}_{\text{loc}}^{1,\frac12}$ hypersurface. We finish the proof by constructing the map $\Psi$. We use an argument essentially due to Ehresmann \cite{ehresmann-1951-infinitesimal-connections-in-differentiable-fiber-spaces}. We include the details since we are working in a low regularity setting.

Since
$$
Df\left(\frac{\nabla f}{|\nabla f|^2}\right)=1,
$$
a uniform smooth approximation gives a smooth vector field $Z \colon U_p \to \mathbb R^n$ such that
$2 \geq (Df)_x(Z(x)) \geq 1/2$ for every $x\in U_p$. Let $\Phi \colon \R \times U_p \supset V \to U_p$ be its (maximally defined) flow.

Let $0 < r < D_\kappa$. Consider the functions $G \colon V \to \R$, defined as $G( t, x ) = f( \Phi(t,x) )$ and $\tau_r \colon V \to \R$, where $(t,x) \mapsto \tau_r(x) \in \R$ is the unique time for which $G( \tau_r(t,x),x ) = r$. Observe that
$$
    \frac{\partial}{\partial t}
    G(t,x)
    \bigg|_{t=s}
    =
    (Df)_{\Phi(s,x)}(Z(\Phi(s,x)))
    \geq
    \frac{1}{2},
$$
so, by the $\mathcal{C}^{1,\frac{1}{2}}_{\text{loc}}$-implicit function theorem, it follows that $(t,x) \mapsto \tau_r(t,x)$ is $\mathcal{C}^{1,\frac{1}{2}}_{\text{loc}}$-regular. Note that the intersection $V_r = ( \R \times f^{-1}(r) ) \cap V$ is an $n$-dimensional $\mathcal{C}^{1,\frac{1}{2}}_{\text{loc}}$-submanifold of $V$, and thus $\tau_r \colon V_r \to \R$ is $\mathcal{C}^{1,\frac{1}{2}}_{\text{loc}}$-regular.

Here $$\Psi \colon U_p \to (0,D_\kappa) \times f^{-1}(r),$$ where $x \mapsto ( f(x), \Phi(\tau_r(x),x) )$, is the required map. Indeed, the inverse of $\Psi$ is given by $(s,y) \mapsto \Phi( \tau_{s}(y), y  )$, where $\tau_s$ is defined above. Since both maps are $\mathcal{C}^{1,\frac{1}{2}}_{\text{loc}}$-regular, the claim follows.
\end{proof}

\begin{proposition}\label{proposition-radial-projection-C1half-diffeo}
Suppose that $( \Omega, k_\Omega )$ is a quasihyperbolic domain. Then, for every $0<r \leq D_2$, the radial projection
$$
\mathcal R_r \colon f^{-1}(r)\to\mathbb S^{n-1}(p,\delta_{\Omega}(p)r),
\quad
x \mapsto p+\delta_{\Omega}(p)r\frac{x-p}{|x-p|},
$$
is a $\mathcal{C}^{1,\frac{1}{2}}$-diffeomorphism.
\end{proposition}

\begin{remark}
    The proof we present below shows more than we need: the radial projection $\mathcal{R}_r$ is bi-Lipschitz on the interval $(0,r_{*})$, where $r_{*} \in (0,\pi)$ is the unique solution of $e^r(\sin r+\cos r)=1$. The bi-Lipschitz conclusion for the range $(0,r_{*})$ extend the corresponding result by Väisälä on the interval $(0,\pi/2)$, see \cite[Theorem~3.11]{vaisala-2007-quasihyperbolic-geometry-of-domains-in-hilbert-spaces}. In fact, when $\Sph^{n-1}(p, \delta_{\Omega}(p)r)$ is equipped with the Euclidean distance, where $0 < r < r_{*}$, the bi-Lipschitz constant of $\mathcal{R}_r$ is bounded from above by
    \begin{align}\label{equation-bi-Lipschitz-constant-radial-projection}
        \max\left\{
            \frac{r}{1-e^{-r}},
            \frac{\pi(e^{r}-1)^2}{r(\sin r+\cos r-e^{-r})}
        \right\}.  
    \end{align}
    The function defined by \eqref{equation-bi-Lipschitz-constant-radial-projection} is continuous and bounded on each interval $(0,r_0]$ for each $r_0 \in (0,r_{*})$. Since we do not need this fact, we instead note that $D_2 < r_{*}$ and $D_{2} \leq D_{\kappa}$ and present a proof of the $\mathcal{C}^{1,\frac{1}{2}}$-diffeomorphism conclusion for the interval $0 < r < D_2$.
\end{remark}

\begin{proof}[Proof of Proposition~\ref{proposition-radial-projection-C1half-diffeo}]
We use the following facts during the proof: for $0 < r < D_2$, the restriction of the smooth radial projection to $f^{-1}(r)$ is
$\mathcal{C}^{1,\frac12}$ and $\ker(D\mathcal R_r)_x=\mathbb R \cdot (x-p) \cap T_xf^{-1}(r)$.

To start, we first suppose that there exists $r_0>0$ such that for all $0 < r < r_0$, $\langle x-p,\nu_x\rangle>0$ for every $x \in f^{-1}(r)$. We claim that this implies that $\mathcal{R}_r$ has full rank and, in fact, that $\mathcal{R}_r$ is a $\mathcal{C}^{1,\frac{1}{2}}$-diffeomorphism. Indeed, each ray $t \mapsto p+tu$ intersects $f^{-1}(r)$ at least once, implying the surjectivity of $\mathcal{R}_r$. To deduce injectivity using the radius bound, if $p+sv$ is such an intersection point, then
$$
\frac{d}{dt}f(p+tu)
=
\frac{\langle u, \nu_{p+tu}\rangle}
{\delta_\Omega(p+tu)}
>0
$$
for every $0 < t \leq s$. Suppose that there were $0 < s_1 < s_2$ such that $f(p+s_1u) = r = f(p+s_2u)$. Then Rolle's theorem yields the existence of $s_1 < t_1 < s_2$ with
\begin{align*}
    0
    =
    \frac{d}{dt}f(p+tu)\biggr|_{t=t_1}
    =
    \frac{\langle u, \nu_{p+t_1u}\rangle}
    {\delta_\Omega(p+t_1u)}.
\end{align*}
This implies that we must have $f(p+t_1u) > r_0$. Since $t \mapsto f(p+tu)$ is increasing on the connected components of $f^{-1}( [0,r_0) )$, the infimum over the values of $t \mapsto f(p+tu)$ after $t_1$ is at least $r_0$ which contradicts the definition of $s_2$. Hence $\mathcal{R}_r$ must be injective.

The definition of $r_0$ also implies that $\mathcal{R}_r$ has full rank, so $\mathcal{R}_r$ being a $\mathcal{C}^{1,\frac{1}{2}}$-diffeomorphism follows.

To finish the proof, we show that for every $0 < r < D_2$ and a quasihyperbolic geodesic $\gamma \colon [0,r] \to (\Omega,k_\Omega)$ joining $p$ to $f^{-1}(r)$, it holds that $\langle x-p, \gamma'(1) \rangle > 0$. This suffices for the claim. Define the Euclidean unit tangent by $T(t)=\gamma'(t)/|\gamma'(t)|$. Then $T(r)=\gamma'(1)/|\gamma'(t)|$ and $\gamma'(t)=\delta_{\Omega}(\gamma(t))T(t)$. Lemma~\ref{lemma-lipschitz-constant-of-qh-geodesics}, after reparametrization, implies $\operatorname{Lip}(T)\leq 1$. Consequently,
$$
\angle(T(t),T(r))=d_{\mathbb S^{n-1}}(T(t),T(r))\le \ell_{\Sph^{n-1}}(T|_{[t,r]})=\int_t^r |T'(u)|\,du\le r-t.
$$
Since $r<\pi$ and cosine is decreasing on $[0,\pi]$, we have 
$$
\langle T(t),T(r)\rangle=\cos\bigl(d_{\Sph^{n-1}}(T(t),T(r))\bigr)\geq\cos(r-t).
$$
Using $\gamma'(t)=\delta_{\Omega}(\gamma(t))T(t)$, we obtain
\begin{align}
\nonumber  \langle x-p, T(r)\rangle
&=
\int_0^r \delta_{\Omega}(\gamma(t))\langle T(t),T(r)\rangle\,dt \\
&\geq
\int_0^r \delta_{\Omega}(\gamma(t))\cos(r-t)\,dt.
\label{eq:main-integral}
\end{align}

To estimate the integral~\eqref{eq:main-integral}, we first consider the case $\pi/2 < r$ and let $t_0 = r-\pi/2$. Then, as~\eqref{eq:qh-estimate-1} gives
\begin{equation}\label{equation-logarithmic-control}
\left|\log\frac{\delta_{\Omega}(\gamma(t))}
{\delta_{\Omega}(\gamma(t_0))}\right|\leq |t-t_0|,
\end{equation}
we conclude that
\begin{equation}\label{eq:entire-interval}
    \delta_{\Omega}(\gamma(t))\cos(r-t)\geq
\delta_{\Omega}(\gamma(t_0))e^{t_0-t}\cos(r-t)
\end{equation}
by applying \eqref{equation-logarithmic-control} separately on the intervals $(0,t_0)$ and $(t_0,r)$, depending on the sign of $\cos$.

 Using \eqref{eq:entire-interval}, we obtain
\begin{align*}
    \langle x-p, T(r)\rangle
&\geq
\delta_{\Omega}(\gamma(t_0))
\int_0^r e^{t_0-t}\cos(r-t)\,dt \\
&=\frac{\delta_{\Omega}(\gamma(t_0))e^{-\pi/2}}{2}
\left[e^r(\sin r+\cos r)-1\right]\\
&\geq
\frac{\delta_\Omega(p)}2
\left(\sin r+\cos r-e^{-r}\right),
\end{align*}
where the last inequality uses 
$\delta_\Omega(\gamma(t_0))\geq e^{-t_0}\delta_\Omega(p)$
and positivity of the bracket. Thus, for every such minimizing
geodesic,
\begin{equation}\label{eq:positive-radial-terminal-tangent}
\langle x-p,T(r)\rangle
\geq
\frac{\delta_\Omega(p)}2
\left(\sin r+\cos r-e^{-r}\right)>0.
\end{equation}
In the simpler case $r\le \pi/2$, we get
\begin{align*}
    \langle x-p,T(r)\rangle
&\geq
\delta_\Omega(p)\int_0^r e^{-t}\cos(r-t)\,dt\\
&=
\frac{\delta_\Omega(p)}2
\left(\sin r+\cos r-e^{-r}\right)
>
0.
\end{align*}
So inequality \eqref{eq:positive-radial-terminal-tangent} holds for every quasihyperbolic geodesic joining $p$ to $f^{-1}(r)$ for every $0 < r < r_{*}$. Since $D_2 < r_{*}$, this completes the proof.
\end{proof}
\begin{proof}[Proof of Theorem~\ref{theorem-c1-regular-spheres}]
    Let $\kappa \leq 2$ be such that $( \Omega, k_\Omega )$ is $\mathrm{CAT}(\kappa)$.
    Proposition~\ref{proposition-radial-projection-C1half-diffeo} shows that, for $0 < r < D_2$, the radial projection $\mathcal R_r \colon f^{-1}(r)\to\mathbb S^{n-1}$ is a $\mathcal{C}^{1,\frac12}$-diffeomorphism. Then $\mathrm{id} \times \mathcal{R}_r \colon (0,D_\kappa) \times f^{-1}(r) \to (0,D_\kappa) \times \Sph^{n-1},$ $(t,x) \mapsto (t, \mathcal{R}_{r}(x))$, is a $\mathcal{C}^{1,\frac12}_{\text{loc}}$-diffeomorphism. Composition with the map $\Psi \colon U_p \to (0,D_\kappa) \times f^{-1}(r)$ from Lemma~\ref{lemma-properties-of-f} concludes the proof.
\end{proof}
\subsection{Angles of quasihyperbolic metrics and Euclidean convexity}
We finish this section with the compatibility of the Alexandrov and Euclidean angles along quasihyperbolic geodesics, with applications to Euclidean convexity. For this, consider a domain $\Omega \subsetneq \R^n$ and local geodesics $\gamma, \theta \colon [0,a] \to ( \Omega, k_\Omega )$ parametrized by (quasihyperbolic) arc length with $\gamma(0) = p = \theta(0)$. Then the \emph{Alexandrov angle} is
\begin{align*}
    \angle( \gamma, \theta ) = \lim_{ t \to 0^{+} } 2 \arcsin\left( \frac{1}{2t} k_\Omega( \gamma(t), \theta(t) ) \right).
\end{align*}
The corresponding \emph{Euclidean angle} is
\begin{align*}
    \angle_{\euc}( \gamma, \theta ) = \arccos\left( \frac{ \langle \gamma'(0), \theta'(0) \rangle }{ \delta_{\Omega}(p)^2 } \right),
\end{align*}
where $\arccos \colon [-1,1] \to [0,\pi]$. The following proposition shows that these angles coincide, extending \cite[Proposition E]{Herron-2020}.
\begin{proposition}\label{proposition-coincidence-of-two-angles}
    Let $( \Omega, k_\Omega )$ be a quasihyperbolic domain. Then, for each pair of quasihyperbolic geodesics $\gamma, \theta \colon [0,a] \to ( \Omega, k_\Omega )$ with $\gamma(0) = p = \theta(0)$, the Alexandrov and Euclidean angles of $\gamma$ and $\theta$ agree.
\end{proposition}
\begin{proof}
    By the differentiability of $\gamma$ and $\theta$ at $0$ and the local bi-Lipschitz equivalence of $k_\Omega$ to the Euclidean distance near $p$, it holds that
\begin{align*}
    \left| k_\Omega( \gamma(t), \theta(t) ) - k_\Omega( p + \gamma'(0) t, \theta(t) ) \right|
    =
    o(t)
\end{align*}
and similarly
\begin{align*}
    \left| k_\Omega( p + \gamma'(0) t, \theta(t) ) - k_\Omega( p + \gamma'(0) t, p + \theta'(0) t )  \right|
    =
    o(t),
\end{align*}
where the small $o$ notation stands for $\lim_{ t \to 0 } o(t)/t = 0$.
Therefore
\begin{align*}
    \frac{1}{2t} k_\Omega( \gamma(t), \theta(t) )
    =
    \frac{1}{2t} k_\Omega( p + \gamma'(0)t, p + \theta'(0)t )
    +
    \frac{ o(t) }{ 2t }.
\end{align*}
Furthermore, as $q \mapsto 1/\delta_{\Omega}(q)$ is continuous and \eqref{eq:qh-estimate-1} holds, there exists a function $\lim_{ t \to 0^{+} } o(1) = 0$, depending on $p$, for which
\begin{align*}
    \frac{ 1 }{ 2t } k_\Omega( p + \gamma'(0)t, p + \theta'(0)t) )
    =
    \frac{ 1 + o(1) }{2 }\frac{ |\gamma'(0)-\theta'(0)| }{ \delta_{\Omega}(p) }
    +
    o(1).
\end{align*}
We conclude that
\begin{align*}
    \angle( \gamma, \theta ) 
    &= \lim_{ t \to 0^{+} } 2 \arcsin\left( \frac{1}{2t} k_\Omega( p + \gamma'(0)t, p + \theta'(0)t ) \right)
    \\
    &= 2 \arcsin\left( \frac{1}{2}\frac{ \left| \gamma'(0) - \theta'(0) \right| }{ \delta_{\Omega}(p) } \right)
    = \angle_{\euc}( \gamma, \theta );
\end{align*}
the last equality follows by the law of cosines. The claimed equality follows.
\end{proof}

For the following statement, define the cotangent function
$$
\operatorname{ct}_{\kappa}(r)
=
\begin{cases}
	\sqrt{\kappa}\cot(\sqrt{\kappa}r), & \kappa>0,\\
	1/r, & \kappa=0,\\
	\sqrt{-\kappa}\coth(\sqrt{-\kappa}r), & \kappa<0.
\end{cases}
$$
Notice that $r \mapsto \operatorname{ct}_\kappa(r)$ is strictly decreasing on the interval $[0,D_\kappa)$.
\begin{theorem}\label{theorem-main-convexity-result}
Let $(\Omega, k_{\Omega})$ be a quasihyperbolic domain that is $\mathrm{CAT}(\kappa)$ for $\kappa \leq 2$. If $\operatorname{ct}_{\kappa}(r) \geq 1$, then the quasihyperbolic ball $B_{k_{\Omega}}(p,r)$ is convex with respect to the Euclidean metric and strictly convex if the inequality is strict. When $\kappa \leq 0$ or $n = 2$, the conclusions hold for $r \leq 1$.
\end{theorem}

\begin{proof}
We fix $0 < r < D_\kappa/2$ for now. Fix distinct point $x,y \in f^{-1}(r)$ which are sufficiently close and write $d=k_{\Omega}(x,y)$. Let $\gamma \colon [0,d]\to\Omega$ be the quasihyperbolic unit-speed geodesic from $x$ to $y$. Define its Euclidean unit tangent by $T(t)=\gamma'(t)/|\gamma'(t)|$. Then 
$$
y-x=\int_0^d \delta_{\Omega}(\gamma(t))T(t)\,dt.
$$

Let $\alpha$ be the angle at $x$ between the quasihyperbolic geodesics from $x$ to $p$ and from $x$ to $y$. The initial Euclidean unit tangent of the reversed radial geodesic $x\to p$ is $-\nu_x$. Moreover, quasihyperbolic and Euclidean angles at $x$ agree by Proposition~\ref{proposition-coincidence-of-two-angles}. Therefore,  $\alpha=\angle(-\nu_x, T(0))$. In case $d+2r < 2D_\kappa$, which is true for every small enough $d$, the $\mathrm{CAT}(\kappa)$ triangle comparison applies to the triangle with vertices $p,x,y$. Let $\bar \alpha$ denote the comparison angle corresponding to $\alpha$. Then $\alpha\le \bar \alpha$ by the comparison estimates. The cosine law for the isosceles triangle with side lengths $r,r,d$ gives
$$
\cos\bar{\alpha}
=
\operatorname{ct}_{\kappa}(r)\operatorname{tn}_{\kappa}(d/2),
$$
where
$$
\operatorname{tn}_{\kappa}(t)
=
\begin{cases}
	\tan(\sqrt{\kappa}t)/\sqrt{\kappa}, & \kappa>0,\\
	t, & \kappa=0,\\
	\tanh(\sqrt{-\kappa}t)/\sqrt{-\kappa}, & \kappa<0.
\end{cases}
$$
Since cosine is decreasing on $[0,\pi]$, $ \langle T(0),\nu_x\rangle = -\cos\alpha \leq -\cos\bar{\alpha}$. 

Let $C$ be the Euclidean arclength parametrization of $\gamma$, and write $\gamma(t)=C(\sigma(t))$, where $\sigma(t)=\ell_{\euc}(\gamma|_{[0,t]})$. It holds that $T(t) = C'( \sigma(t) )$ so Lemma~\ref{lemma-lipschitz-constant-of-qh-geodesics} implies $\operatorname{lip}( T )(t) = \operatorname{lip}( C' )( \sigma(t) )|\sigma'(t)| \leq 1$. Therefore $$|T(t)-T(0)| \leq t.$$ Using the quasihyperbolic estimate and letting $d$ be so small that $\bar{\alpha} \leq \pi/2$, we deduce that
    \begin{align*}
\langle y-x,\nu_x\rangle
&\leq
\int_0^d \delta_{\Omega}(\gamma(t))(-\cos\bar\alpha+t)\,dt \\
&\leq
\delta_{\Omega}(x)
\left[
-\cos\bar\alpha(1-e^{-d})
+
\int_0^d te^t\,dt
\right] \\
&=
\delta_{\Omega}(x)
\left[
-\cos\bar\alpha(1-e^{-d})
+
(d-1)e^d+1
\right].
\end{align*}
An upper bound $\leq 0$ for the last term is equivalent to
\begin{align*}
    f(d) = \frac{ ( (d-1)e^{d} + 1 ) }{ \operatorname{tn}_{\kappa}(d/2)(1-e^{-d}) } \leq \frac{ \cos( \bar{\alpha} ) }{ \operatorname{tn}_{\kappa}(d/2) } = \operatorname{ct}_{\kappa}(r).
\end{align*}
Since $\lim_{d \to 0^{+}} f(d) = 1$, we need only consider the radii $r$ for which $\operatorname{ct}_{\kappa}( r ) > 1$. Equivalently, the radii with $0 < r < r_0 = \operatorname{ct}_\kappa^{-1}(1)$. Having fixed $r< r_0$, the construction implies the existence of $d_0>0$ where $\langle y-x, \nu_x \rangle < 0$ whenever $x,y \in f^{-1}(r)$ and $k_{\Omega}(x,y) < d_0$. This implies strict convexity for the $\mathcal{C}^{1}$-regular domain $B_{k_\Omega}(p,r)$ and its closure $\overline{B}_{k_\Omega}(p,r)$.

Finally, suppose that $\operatorname{ct}_{\kappa}(r)=1$ and $r_j\uparrow r$. Since every $\bar{B}_{k_{\Omega}}(p,r_j)$ is Euclidean convex, so is $B_{k_\Omega}(p,r)$ as a nested union of convex sets. Thus so is its closure $\overline{B}_{k_\Omega}(p,r)$.

When $n = 2$, we use the fact that $\overline{B}_{k_\Omega}(p,\pi/2)$ is $\mathrm{CAT}(0)$. With this observation, the analysis above can easily be adapted to obtain strict convexity for each $r<1$ and convexity for $r = 1$.
\end{proof}

\begin{proof}[Proof of Theorem~\ref{theorem-small-balls-euclidean-convex}]
This is a reformulation of Theorem~\ref{theorem-main-convexity-result}.
\end{proof}

\section{Examples}\label{section-examples}

\subsection{Sharpness of the curvature upper bound}\label{section-sharpness-of-the-curvature-bound}
In this section, we prove Theorem \ref{thm:sharp-cat2}. Let $n \geq 3$ and consider $P = \{e_1, -e_1\}$. Let $\Omega = \R^n \setminus P$. 

The proof has two parts. First, we prove that $Q = \{0\} \times \R^{n-1} \subset \Omega$ is $D_2$-convex, i.e., if $x,y \in Q$ with $0 < k_\Omega(x,y) < D_2$, then the quasihyperbolic geodesic $\gamma \colon [a,b] \to ( \Omega, k_\Omega )$ joining $x$ to $y$ is contained in $Q$. Using this, we prove that the length distance on $Q$ is obtained from a smooth Riemannian metric. Afterwards, an explicit computation demonstrates that the sectional curvature at the origin $0 \in Q$ is exactly $2$. This implies that $( \Omega, k_\Omega )$ cannot be locally $\mathrm{CAT}(\kappa)$ for any $\kappa < 2$ by Theorem~\ref{thm:riemannian-local-CAT}.

We start with the claim that $Q \subset ( \Omega, k_\Omega )$ is $D_2$-convex. We first note that $T \colon ( \Omega, k_\Omega ) \to ( \Omega, k_\Omega )$, where $T(x,y) = (-x,y)$, is an isometry. Hence, if there were a non-trivial component $I \subset \gamma^{-1}( \{ x < 0 \} )$, the curve obtained by concatenating $T \circ \gamma|_{I}$ to $\gamma|_{ [a,b] \setminus I }$ would yield another quasihyperbolic geodesic joining $a$ to $b$. This contradicts uniqueness. Hence no such component $I$ exists. By symmetry, it follows that $\gamma$ lies in $Q$.

Since $Q \subset ( \Omega, k_\Omega )$ is $D_2$-convex, it follows that $Q$ is $\mathrm{CAT}(2)$ when equipped with the restriction of $k_\Omega$. The distance on $Q$ is obtained from a smooth Riemannian metric. In fact, consider $L \colon \R^{n-1} \to Q$, where $L(x) = (0,x)$. When $\R^{n-1}$ is equipped with the metric tensor
\begin{equation*}
    g = \frac{ 1 }{ 1 + r^2 }\left( dr^2 + r^2 g_{\mathbb{S}^{n-2}}^2 \right) = \frac{ r^2 }{ 1 + r^2 }\left( \frac{1}{r^2} dr^2 + g_{\mathbb{S}^{n-2}}^2 \right),
\end{equation*}
defined in spherical coordinates, the map $L \colon ( \R^{n-1}, g ) \to ( Q, k_\Omega )$ becomes a path isometry. Next, consider the coordinates $u = \sinh^{-1}(r)$ so $dr = \cosh(u) \,du$ and $\cosh^2(u) = 1+r^2$. Hence
\begin{align*}
    g_{\text{affine}}
    \coloneqq
    \frac{ 1 }{ \cosh^2(u) }
    \left(
        \cosh^2(u) du^2 + \sinh^2(u) g_{\mathbb{S}^{n-2}}^2 
    \right)
    =
    du^2 + \tanh(u)^2 g_{\mathbb{S}^{n-2}}^2.
\end{align*}
Here $\tanh(u) = u + O(u^3)$ so we may extend $g_{\text{affine}}$ as the Euclidean inner product to $u=0$. We abuse notation and denote the induced metric on $\mathbb{R}^{n-1}$ by $g_{\text{affine}}$. Standard sectional curvature formulas for radial metrics show that for a plane defined by a radial and a tangential direction, the sectional curvature of $g_{\text{affine}}$ at $(u,e_1)$ is
\begin{align*}
    K_{g_{\text{affine}}}((u,e_1),\Pi_{\text{radial}}) = - \frac{ \tanh''(u) }{ \tanh(u) } = \frac{ 2 }{ \cosh^2(u) } = \frac{ 2 }{ 1 + r^2 },
\end{align*}
see e.g. \cite[page 27, Remark 7.9 (3)]{Bishop-ONeill-1969-manifolds-of-negative-curvature}. As $r \rightarrow 0^{+}$, this converges to $2$. Therefore $( Q, k_\Omega )$ is locally $\mathrm{CAT}(\kappa)$ only for $\kappa \geq 2$ by Theorem~\ref{thm:riemannian-local-CAT}, and hence the $\mathrm{CAT}(2)$ conclusion cannot be improved when $n \geq 3$.

\subsection{A non-convex CAT(0) example}\label{example:nonconvex-cat0-example}
In this section, we construct an example of a $\CAT(0)$ quasihyperbolic domain for which $\Omega$ is non-convex, concave, and smooth. We claim that the domain $\Omega = \R^{n-2} \times D$ for
$$
D=\{(s,r)\in\R^2\colon r<s^2\}.
$$
suffices for $n \geq 3$.

We express the points of $\Omega$ as $(y,s,r)$, where
$y\in\R^{n-2}$ and $r<s^2$. The map
$(y,s,r)\mapsto(y,s,r-s^2)$ is a diffeomorphism from $\Omega$ onto the lower half-space $\R^{n-1}\times(-\infty,0)$. Hence $\Omega$ is contractible and, in particular, simply connected. Thus, by Cartan--Hadamard theorem, it suffices to prove that $\Omega$ is locally $\mathrm{CAT}(0)$. By Theorem~\ref{thm:riemannian-local-CAT}, it suffices to prove that $\delta_{\Omega}$ is smooth and derive a sectional curvature upper bound of zero for $\delta_{\Omega}^{-2} g_{\euc}$.

We describe the domain and $\delta_\Omega$. First, the complement $C = \R^n \setminus \Omega$ is convex by the convexity of $s \mapsto s^2$. Second, the outer normal of $\partial C = \R^{n-2} \times \partial U$ is
$$
    \nu(p(y,s))=\nu(s)
    =
    \frac{(0,\ldots,0,2s,-1)}{\sqrt{1+4s^2}}.
$$
We consider the nearest point projection $\pi_{\Omega} \colon \Omega \to \partial \Omega$ and $\Phi \colon \Omega \to \partial \Omega \times (0,\infty)$, where $\Phi(x) = ( \pi_{\Omega}(x), \delta_\Omega(x) )$ for $x \in \Omega$. The normal map
$$
    \Psi\colon\partial\Omega\times(0,\infty)\longrightarrow\R^n,
    \quad
    \Psi(p,t)=p+t\nu(p)
$$
is the inverse of $\Phi$. Indeed, for $p=p(y,s)\in\partial \Omega$, $t>0$, and $q\in C$, the supporting
hyperplane property gives $\langle \nu(s), q-p\rangle\leq0$.
Consequently,
$$|\Psi(p,t)-q|^2=t^2+|q-p|^2-2t\langle\nu(s),q-p\rangle\ge t^2$$
with equality if and only if $q=p$. Hence $\Psi(p,t)\in\Omega$. Moreover, $p$ is the unique nearest point in
$\partial \Omega$ to $\Psi(p,t)$, and $\delta_\Omega(\Psi(p,t))=t$. In fact, $\Phi \circ \Psi = \mathrm{id}_{ \partial \Omega \times (0,\infty) }$ and $\Psi \circ \Phi = \mathrm{id}_{ \Omega }$.

An orthonormal basis of the tangent space $T_{(y,s)}\partial \Omega$ in $\partial \Omega = \mathbb{R}^{n-2} \times \partial \Omega$ is given by an orthonormal basis of $\R^{n-2} \times \{0\}$ and
$$
    E(s)
    =
    \frac{(0,\ldots,0,1,2s)}{\sqrt{1+4s^2}}.
$$
We also denote
$$
    \kappa(s) = \frac{2}{(1+4s^2)^{3/2}},
$$
where $\kappa(s)$ is the curvature of $t \mapsto (t,t^2)$ at $s$.

For every vector $v$ tangent to the $\R^{n-2}$-factor, the differential of $\Psi$ satisfies
\begin{align}
	D\Psi_{(p(y,s),t)}(v,0)
	&=v,
	\label{eq:differential-of-psi-1}\\
	D\Psi_{(p(y,s),t)}(E(s),0)
	&=(1+t\kappa(s))E(s),
	\label{eq:differential-of-psi-2}\\
	D\Psi_{(p(y,s),t)}(0,1)
	&=\nu(s).
	\label{eq:differential-of-psi-3}
\end{align}
It follows that the differential of $\Psi$ has full rank everywhere and thus $\Psi$ and its inverse are diffeomorphisms. In particular, $\delta_\Omega$ is smooth. Differentiating the identity $\delta_\Omega(\Psi(p(y,s),t))=t$ and using \eqref{eq:differential-of-psi-1}, \eqref{eq:differential-of-psi-2}, and \eqref{eq:differential-of-psi-3} for $D\Psi$ gives
\begin{equation}\label{eq:nabla-delta-omega}
    \nabla\delta_\Omega(\Psi(p(y,s),t))=\nu(s).
\end{equation}

Fix $x=\Psi(p(y,s),t)$. To bound the eigenvalues of the Hessian $(\Hess\delta_\Omega)_x$, we need to differentiate the identity \eqref{eq:nabla-delta-omega}. For this purpose, we need to differentiate $\nu$ as well. Since $\nu$ is independent of $y$,
$$
(D\nu)_{p(y,s)}(v)=0
$$
for every vector $v$ tangent to the $\R^{n-2}$-factor. Moreover,
direct differentiation gives
$$
(D\nu)_{p(y,s)}(E(s))
=
\kappa(s)E(s).
$$
Now, using the identities \eqref{eq:differential-of-psi-1}-\eqref{eq:nabla-delta-omega} shows that the Euclidean
symmetric endomorphism associated with
$(\Hess\delta_\Omega)_x$ has eigenvalue $\frac{\kappa(s)}{1+t\kappa(s)}$ in the direction $E(s)$ and vanishes on $E(s)^\perp$. We obtain from Lemma~\ref{lem:smooth-distance-curvature} that
\begin{align*}
    K_g((p(y,s),t),\Pi)
    &=
    t\operatorname{tr}_{\Pi}(\Hess\delta_\Omega)_{(p(y,s),t)} - 1
    \leq
    \frac{ t \kappa(s) }{ 1 + t\kappa(s) } - 1
    <
    0
\end{align*}
for every $2$-plane. This  implies that $(\Omega,k_\Omega)$ is locally $\CAT(0)$. The proof is complete.

\section*{Acknowledgements}
We thank Nicola Cavallucci, Pekka Koskela, Alexander Lytchak and Kai Rajala for feedback on an earlier version of this manuscript. We also thank Stefan Wenger and Pekka Pankka for discussions on the topic.

\providecommand{\bysame}{\leavevmode\hbox to3em{\hrulefill}\thinspace}
\providecommand{\MR}{\relax\ifhmode\unskip\space\fi MR }
\providecommand{\MRhref}[2]{%
  \href{http://www.ams.org/mathscinet-getitem?mr=#1}{#2}
}
\providecommand{\href}[2]{#2}

\end{document}